\documentclass[reqno]{amsart}

\usepackage{amsmath}
\usepackage{amssymb}
\usepackage{amsfonts}
\usepackage{comment}
\usepackage[all,cmtip]{xy}
\usepackage{mathrsfs}  
\usepackage{fullpage}
\usepackage{xcolor}
\allowdisplaybreaks

\theoremstyle{definition}
\newtheorem{theorem}{Theorem}[section]

\newtheorem{corollary}[theorem]{Corollary} 
\newtheorem{definition}[theorem]{Definition} 
\newtheorem{example}[theorem]{Example}
\newtheorem{lemma}[theorem]{Lemma}

\newtheorem{question}[theorem]{Question}
\newtheorem{remark}[theorem]{Remark}

\DeclareMathOperator\Aut{Aut}

\DeclareMathOperator\gr{gr}
\DeclareMathOperator\GKdim{GKdim}

\DeclareMathOperator\Kdim{Kdim}

\DeclareMathOperator\Maxspec{MaxSpec}

\DeclareMathOperator\der{Der}
\newcommand\derh{\mathrm{Der}^\mathrm{H}}

\DeclareMathOperator\pder{PDer}

\DeclareMathOperator\lnd{LND}

\newcommand\lndds{\lnd_d^*}

\newcommand\lndh{\lnd^\mathrm{H}}

\newcommand\plnd{\mathrm{PLND}}
\newcommand\plndd{\plnd_d}
\newcommand\plnds{\plnd^*}
\newcommand\plndp{\plnd_{\pcnt}}
\newcommand\plnddp{\plnd_{d,\pcnt}}
\newcommand\plndh{\plnd^\mathrm{H}}
\newcommand\plndds{\plnd_d^*}
\newcommand\plndps{\plnd_{\pcnt}^*}

\newcommand\plnddps{\plnd_{d,\pcnt}^*}
\newcommand\plnddh{\plnd_d^\mathrm{H}}

\newcommand\ml{\mathrm{ML}}

\newcommand\mlds{\ml_d^*}

\newcommand\pml{\mathrm{PML}}
\newcommand\pmls{\pml^*}
\newcommand\pmld{\pml_d}

\newcommand\pmlds{\pml_d^*}

\renewcommand\int{\mathrm{int}}

\newcommand\cnt{\mathcal Z}
\newcommand\pcnt{\cnt_P}
\newcommand\inv{^{-1}}
\newcommand\niso{\ncong}
\newcommand\iso{\cong}
\newcommand\tensor{\otimes}

\newcommand\cC{\mathcal C}

\newcommand\cP{\mathcal P}
\newcommand\cS{\mathcal S}

\newcommand{\kk}{\Bbbk}
\newcommand{\NN}{\mathbb N}
\newcommand{\RR}{\mathbb R}

\newcommand{\CC}{\mathbb C}
\newcommand{\VV}{\mathbb V}

\newcommand\fm{\mathfrak m}

\renewcommand\H{\mathrm{H}}

\begin{document}

\title{The Zariski cancellation problem for Poisson algebras}

\author[Gaddis]{Jason Gaddis}
\address{Miami University, Department of Mathematics, 301 S. Patterson Ave., Oxford, Ohio 45056} 
\email{gaddisj@miamioh.edu}

\author[Wang]{Xingting Wang}
\address{Department of Mathematics, Howard University, Washington DC, 20059} 
\email{xingting.wang@howard.edu}

\date{\today}
\begin{abstract}
We study the Zariski cancellation problem for Poisson algebras asking whether $A[t]\cong B[t]$ implies $A\cong B$ when $A$ and $B$ are Poisson algebras. We resolve this affirmatively in the cases when $A$ and $B$ are both connected graded Poisson algebras finitely generated in degree one without degree one Poisson central elements and when $A$ is a Poisson integral domain of Krull dimension two with nontrivial Poisson bracket. We further introduce Poisson analogues of the Makar-Limanov invariant and the discriminant to deal with the Zariski cancellation problem for other families of Poisson algebras.
\end{abstract}

\maketitle

\setcounter{tocdepth}{1}
\tableofcontents

\section{Introduction}
\label{sec.intro}

It is both surprising and interesting to know that many fundamental questions about the geometry and symmetries of the affine space $\mathbb A^n$, which is a basic object in algebraic geometry, still remain open in the 21st century. Among those open questions is the cancellation property, which was asked by Zariski in the following sense. 
\begin{question}[(Zariski Cancellation Problem)]
Does an isomorphism $Y\times \mathbb A^1\cong A^{n+1}$ imply an isomorphism $Y\cong \mathbb A^n$, for any affine variety $Y$?
\end{question}
The interested reader is directed to the survey by Gupta for further background on this problem \cite{gupta}. In recent years, several works have extended the Zariski cancellation problem from commutative algebraic geometry to noncommutative projective algebraic geometry, where the cancellation property is studied for certain types of Artin-Schelter regular algebras, which are noncommutative graded analogues of commutative polynomial rings; see \cite{BZ1,BZ2,LWZ1,LWZ2}.

In this paper, we extend the original Zariski cancellation problem in a different direction. Instead of generalizing the affine variety $Y$ into a noncommutative algebraic variety, which is represented by a noncommutative algebra as its coordinate ring, we assume $Y$ to have extra structure, namely we assume that there exists a bivector $\pi\in \bigwedge^2(TY)$ satisfying a vanishing Schouten-Nijenhuis bracket $[\pi,\pi]=0$. In algebraic terms, $Y$ is an affine Poisson variety whose coordinate ring turns out to be a commutative Poisson algebra. Therefore, we are interested in the following question.

\begin{question}[(Zariski Cancellation Problem for Poisson Algebras)]
When is a Poisson algebra $A$ cancellative? That is, when does an isomorphism of Poisson algebras $A[t] \iso B[t]$ for another Poisson algebra $B$ imply an isomorphism $A \iso B$ as Poisson algebras?
\end{question}

The notion of the Poisson bracket, first introduced by  Sim{\'e}on Denis Poisson, arises naturally in Hamiltonian mechanics and differential geometry. Poisson algebras have become deeply entangled with non-commutative geometry, integrable systems, and topological field theories.
They are essential in the study of the noncommutative discriminant \cite{BY,NTY} and representation theory of noncommutative algebras \cite{WWY1,WWY2}. In addition, there has been renewed interest in enveloping algebras of Poisson algebras \cite{LOW,LWZ3}.

Recently, Adjamagbo and van den Essen \cite{avdE} proved that the famous Jacobian conjecture for polynomial algebras has an equivalent statement for Poisson algebras. As pointed out by van den Essen \cite{VDE}, the Zariski cancellation problem (especially in dimension two) is closely related to the Jacobian conjecture. Therefore, one of our motivations is to study the Zariski cancellation problem for Poisson algebras with a potential link to the above-mentioned Poisson version of the Jacobian conjecture.

While we are searching for general methods and techniques in this direction, we find many theories developed by Bell and Zhang in the (associative) noncommutative setting \cite{BZ1,BZ2} can be adapted to the Poisson setting. Their work relies both on the idea of the noncommutative discriminant \cite{CPWZ1,CPWZ2} and the Makar-Limanov invariant \cite{ML}. Other work on noncommutative Zariski cancellation has focused on path algebras and their quotients \cite{Gpath,LWZ1} as well as a Morita invariance version \cite{LWZ2}.

In Section \ref{sec.background}, we provide background on Poisson algebras and related concepts that are critical to our study. Section \ref{sec.zariski} continues this and introduces the Zariski cancellation problem for Poisson algebras, where we discuss various versions of the Poisson cancellation property and the relations between them. In Section \ref{sec.graded} we restrict our attention to graded Poisson algebras and establish a graded Poisson version of the Zariski cancellation problem (Theorem \ref{thm.zariski1}). This is a consequence of the following theorem, which is an analogue of a result of Bell and Zhang in the graded (associative) algebra setting \cite{BZ1}.

\begin{theorem}[(Theorem \ref{thm.iso})]
\label{thm.intro1}
Let $A$ and $B$ be two connected graded Poisson algebras finitely generated in degree one. If $A\cong B$ as ungraded Poisson algebras, then $A\cong B$ as graded Poisson algebras. 
\end{theorem}
\noindent This result is then applied to the family of skew quadratic Poisson algebras to obtain all possible isomorphisms between them based on the coefficient matrices given by their Poisson brackets (Theorem \ref{thm.skewP}). 

Much of our remaining work makes use of the {\it Poisson center} and this turns out to be a suitable replacement for the (algebra) center used in \cite{BZ2}. We study the Poisson center and its implications for Poisson cancellation in Section \ref{sec.artcnt}.

\begin{theorem}
\label{thm.intro2}
Let $A$ be a Poisson algebra.
\begin{enumerate}
    \item (Corollary \ref{lemm.artincan}) If $A$ is noetherian with artinian Poisson center, then $A$ is Poisson cancellative.
    \item (Theorem \ref{thm.univ}) If $A$ has trivial Poisson center, then $A$ is Poisson cancellative.
\end{enumerate}
\end{theorem}
\noindent Theorem \ref{thm.intro2} (2) can be applied to show that Poisson integral domains of Krull dimension two which have nontrivial Poisson brackets are Poisson cancellative (Corollary \ref{cor.cancel1}). Here the non-triviality of the Poisson bracket plays an essential role since by \cite{Fi,Da} there are commutative domains of Krull dimension two that are not cancellative. 

It is well-known that locally nilpotent derivations are important in the study of cancellation. This is the crux of the work of Makar-Limanov \cite{ML}. Bell and Zhang exploit these ideas to great effect in their work on cancellation \cite{BZ2}. We further these ideas 
by establishing a connection between locally nilpotent Poisson derivations and Poisson cancellation in Section \ref{sec.ML}.

\begin{theorem}[(Theorem \ref{thm.strongplndcan})]
\label{thm.intro3}
Assume $\kk$ is a field of characteristic zero. Let $A$ be an affine Poisson domain over $\kk$ with finite Krull dimension. If $A$ has no nontrivial locally nilpotent Poisson derivations, then $A$ is Poisson cancellative.
\end{theorem}
\noindent In positive characteristic, we can prove a similar result by replacing Poisson derivations by higher Poisson derivations introduced by Launois and Lecoutre \cite{LL}.

In Section \ref{sec.pdisc} we introduce the Poisson discriminant as well as the notion of effectiveness for these discriminants. We show that effectiveness controls the locally nilpotent Poisson derivations, which in the noncommutative setting was first observed by Bell and Zhang \cite{BZ2} where discriminants are defined for noncommutative algebras that are module-finite over their centers.

\begin{theorem}[(Theorem \ref{thm.prigid})]
Let $A$ be an affine Poisson domain with affine Poisson center. If the Poisson discriminant exists and is effective either in $A$ or its Poisson center, then $A$ is Poisson cancellative.  
\end{theorem}

It is important to mention that for Poisson algebras in characteristic zero their Poisson centers are usually not large enough for us to emulate the definition of discriminants for noncommutative algebras by simply replacing the algebraic center by Poisson center. Hence, we follow the idea in \cite[\S 2]{LWZ2} to introduce the notion of Poisson discriminant from a representation-theoretic point of view. This is leveraged to study a variety of cancellation results related to Poisson algebras for which we can identify a discriminant relative to some property of Poisson algebras. We give two specific examples: 
one is the affine space $\mathbb A^3$ with a nontrivial unimodular Poisson bracket (Example \ref{ex.jacdet}) and the other is derived from the Poisson order on the center of a three-dimensional PI Sklyanin algebra (Example \ref{ex.Skly3}). 

Finally, in Section \ref{S:RQ} we discuss relations between various concepts we introduce in this paper related to the Zariski cancellation problem for Poisson algebras. We further post several open questions that are continuations of the topics that are covered in this paper.

\noindent\textbf{Acknowledgements.}
Part of this research work was done during the second author's visit to the Department of Mathematics at Miami University in March 2019. He is grateful for the first author's invitation and wishes to thank Miami University for its hospitality. The authors would also like to thank Jason Bell and James Zhang for helpful conversations and suggestions.

\section{Preliminaries}
\label{sec.background}
Throughout the paper, we work over a base field $\kk$. Many of our results still work when $\kk$ is only a commutative domain. We will leave the reader to figure out these cases. 

\subsection{Poisson algebras}
A {\sf Poisson algebra} (over $\kk$) is a commutative $\kk$-algebra $A$
equipped with a bilinear product $\{,\}: A \times A \rightarrow A$,
called a {\sf Poisson bracket}, such that $A$ is a Lie algebra under $\{,\}$ 
and the map $\{a, -\}: A \rightarrow A$ is a $\kk$-derivation of $A$ for all $a \in A$. 
When $A$ is a Poisson algebra, the {\sf Poisson center} of $A$ is denoted by 
\[\pcnt(A) = \{ z \in A : \{z,a\}=0 \text{ for all } a \in A \}.\]
We say $A$ has {\sf trivial Poisson center} if $\pcnt(A)=\kk$.
A {\sf homomorphism} $\phi:A \rightarrow B$ of Poisson algebras is an algebra homomorphism  such that for all $a_1,a_2 \in A$, $\phi(\{a_1,a_2\}_A) = \{\phi(a_1),\phi(a_2)\}_B$.
An {\sf isomorphism} of Poisson algebras is a bijective Poisson homomorphism. 
An ideal $I$ of a Poisson algebra $A$ is a {\sf Poisson ideal} if $\{I,A\} \subset I$.
If $I$ is a Poisson ideal of $A$, then $A/I$ is a Poisson algebra with bracket $\{a+I,b+I\}=\{a,b\}+I$ for all $a+I,b+I \in A/I$.

The {\sf Gelfand-Kirillov (GK) dimension} of an affine Poisson algebra $A$ (or more generally, a $\kk$-affine associative algebra $A$) is defined to be
\[\GKdim\, A=\underset{V}{\text{sup}}\left(\overline{\lim_{n\rightarrow \infty}} \log_n\, \dim_\kk\, (V^n)\right),\]
where $V$ varies over all finite-dimensional $\kk$-vector subspaces of $A$. If $A$ is affine commutative, then $\GKdim\, A=\Kdim\, A$ \cite[Theorem 4.5]{GL}, where $\Kdim\, A$ denotes the {\sf Krull dimension} of $A$.

\subsection{Locally nilpotent (Poisson) derivations}
Denote the space of ($\kk$-)derivations (resp. locally nilpotent ($\kk$-)derivations)
of an algebra $A$ by $\der(A)$ (resp. $\lnd(A)$). A {\sf higher derivation (or Hasse-Schmidt derivation)} on $A$ is a sequence of $\kk$-linear endomorphisms $\partial:=\{\partial_i\}_{i=0}^\infty$ such that
$$\partial_0={id}_A\quad \text{and}\quad \partial_n(ab)=\sum_{i=0}^n \partial_i(a)\partial_{n-i}(b)\ \text{for all $a,b\in A$ and all $n\ge 0$}.$$
The collection of higher derivations of $A$ is denoted by $\derh(A)$. 
A higher derivation $\partial$ is called {\sf iterated} if $\partial_i\partial_j=\binom{i+j}{i}\partial_{i+j}$ for all $i,j\ge 0$. 
A higher derivation $\partial$ is called {\sf locally nilpotent} if 
\begin{enumerate}
\item for all $a\in A$ there exists $n\ge 0$ such that $\partial_i(a)=0$ for all $i\ge n$,
\item the map $G_{\partial,t}: A[t]\to A[t]$ defined by $a\mapsto \sum_{i=0}^\infty \partial_i(a)t^i$ and $t\mapsto t$, for all $a\in A$, is an algebra automorphism of $A[t]$.
\end{enumerate}
The collection of locally nilpotent higher derivations of $A$ is denoted by $\lndh(A)$. 

Now let $A$ be a Poisson algebra. A derivation $\alpha$ of $A$ is called a {\sf Poisson derivation} if 
\[ \alpha(\{a,b\}) = \{ \alpha(a),b \} + \{ a, \alpha(b)\} \quad\text{for all $a,b \in A$.} \]
We denote the space of Poisson derivations of $A$ by $\pder(A)$
and the space of locally nilpotent Poisson derivations of $A$ by $\plnd(A)$.
Thus, we have the inclusions $\pder(A) \subset \der(A)$ and $\plnd(A) \subset \lnd(A)$. 
A {\sf higher Poisson derivation} on $A$ is a higher derivation $\partial:=\{\partial_i\}_{i=0}^\infty$ of $A$ such that
\[ \partial_n(\{a,b\})=\sum_{i=0}^n\,\left\{\partial_i(a),\partial_{n-i}(b)\right\} \quad\text{for all $a,b \in A$ and all $n\ge 0$.} \]
If, in addition, $\partial$ is locally nilpotent and the map $G_{\partial,t}$ defined above is a Poisson algebra automorphism, then $\partial$ is said to be a {\sf locally nilpotent higher Poisson derivation}.
We denote the collection of locally nilpotent higher Poisson derivations of $A$ by $\plndh(A)$.

\subsection{Extensions of Poisson algebras}
If $A$ and $B$ are Poisson algebras, then there is a natural Poisson structure on $A \tensor B$ defined by extending linearly the bracket
\begin{align} \label{eq.cext}
\{ a_1 \tensor b_1, a_2 \tensor b_2 \} = \{a_1,a_2\} \tensor b_1b_2 + a_1a_2 \tensor \{b_1,b_2\}
\quad\text{for all } a_1 \tensor b_1,a_2 \tensor b_2 \in A \tensor B.
\end{align}
Thus, the polynomial algebra $A[t] \iso A \tensor \kk[t]$
has a natural bracket, extending the bracket on $A$, defined by setting $\{t,a\}=0$ for all $a \in A$.
We will often identify $A$ and $B$ with their images under the obvious embeddings $A \hookrightarrow A \tensor B$ and $B \hookrightarrow A \tensor B$, respectively.

The following lemma will be used frequently in this paper without mentioning it explicitly.

\begin{lemma}
\label{lem.cnt}
Let $A$ be a Poisson algebra
and let $R$ be an affine Poisson algebra with trivial Poisson bracket.
Then $\pcnt(A \tensor R) = \pcnt(A) \tensor R$.
\end{lemma}
\begin{proof}
It is clear from \eqref{eq.cext} that $\pcnt(A) \tensor R \subset A \tensor R$. 
Let $\sum_{i=1}^n a_i \tensor r_i \in A \tensor R$ with all $r_i$'s being linearly independent over $\kk$.
Assume some $a_k \notin \pcnt(A)$.
Then there exists $a' \in A$ such that $\{a_k,a'\} \neq 0$. Thus,
\[ \left\lbrace \sum_{i=1}^n a_i \tensor r_i , a' \tensor 1 \right\rbrace
	= \sum_{i=1}^n \{ a_i \tensor r_i, a' \tensor 1 \}
	= \sum_{i=1}^n \{ a_i, a'\} \tensor r_i.
\]
It follows that the $r_k$ component is nonzero and so
$\sum_{i=1}^n a_i \tensor r_i \notin \pcnt(A \tensor R)$.
\end{proof}

Also there is a natural Poisson structure on $A\bigoplus B$ defined by extending linearly the bracket 
\begin{align} \label{eq.dext}
\{ a_1 \oplus b_1, a_2 \oplus b_2 \} = \{a_1,a_2\} \oplus \{b_1,b_2\}
\quad\text{for all } a_1\oplus b_1, a_2 \oplus b_2 \in A \bigoplus B.
\end{align}
Suppose $e$ is an idempotent of a commutative algebra $A$. Then $A$ can be decomposed into a direct sum $A=Ae\oplus A(1-e)$ of two algebras $Ae$ and $A(1-e)$. If $A$ is a Poisson algebra, then $Ae$ and $A(1-e)$ are Poisson algebras with Poisson bracket inherited from $A$.

Let $A$ be a Poisson algebra and let $\alpha$ be a Poisson derivation of $A$. A linear map $\delta:A\to A$ is called a {\sf Poisson $\alpha$-derivation} if it satisfies
\begin{enumerate}
\item $\delta(ab)=\delta(a)b+a\delta(b)$;
\item $\delta(\{a,b\})=\{\delta(a),b\}+\{a,\delta(b)\}+\alpha(a)\delta(b)-\delta(a)\alpha(b)$,
\end{enumerate}
for all $a,b\in A$. Let $\delta$ be a Poisson $\alpha$-derivation of $A$. The {\sf Poisson-Ore extension} $A[z;\alpha,\delta]_P$ is the polynomial ring $A[z]$ with Poisson bracket
\[ 
\{a,b\} = \{a,b\}_A, \qquad
\{z,a\} = \alpha(a)z +\delta(a)\quad \text{for all } a,b \in A.
\]
For simplicity, we write $A[t]=A[t;0,0]_P$ and $A[t_1,\dots,t_n]=A[t_1]\cdots[t_n]$ for any $n\ge 1$.

\subsection{Filtered and graded Poisson algebras}
An {\sf ascending Poisson $\NN$-filtration} $F$ on a Poisson algebra $A$ is a collection of subspaces $\{F_i A\}_{i \geq 0}$ satisfying
\begin{enumerate}
    \item $F_i A \subseteq F_{i+1} A$,
    \item $\bigcup_{i\geq 0} F_i A = A$,
    \item $F_i A \cdot F_j A \subseteq F_{i+j} A$, and 
    \item $\{F_i A, F_j A\} \subseteq F_{i+j} A$, for all $i, j \geq 0$.
\end{enumerate}
If $A$ is a Poisson algebra with a Poisson $\NN$-filtration $F$, then the {\sf associated graded Poisson algebra $\gr_F(A)$} is defined as
\[ \gr_F(A) := \bigoplus_{m \geq 0} F_m A/ F_{m-1} A\]
with $F_{-1}(A)=0$.
We drop the subscript if the filtration is implied. Similarly, we can define a {\sf descending Poisson $\NN$-filtration} $F$ on a Poisson algebra $A$.

\begin{example}
(1) Let $A$ be the first Poisson Weyl algebra.
That is, $A=\kk[x,y]$ with Poisson bracket $\{x,y\}=1$.
Set the ascending filtration $F$ on $A$ by defining $x$ and $y$ to have degree 1
with $F_n A$ being the span of all monomials with degree at most $n$.
Then $\gr_F A$ is the polynomial algebra $\kk[x,y]$ with trivial Poisson bracket.

(2) Any Poisson ideal $I$ of a Poisson algebra $A$ defines a descending $I$-adic Poisson $\NN$-filtration on $A$ by setting $F_i A = I^i$ for all $n\ge 0$ ($I^0=A$). We denote the corresponding associated graded Poisson algebra by $\gr_I(A)$.
\end{example}

Let $A$ be a Poisson algebra. When $\gr_F A=A$ for some Poisson $\NN$-filtration $F$, then $A$ is said to be {\sf $\NN$-graded} and in this case we generally write $A_i$ for $F_i A/F_{i-1}A$. Additionally, if $F_0A=\kk$ we say $A$ is {\sf connected graded} and we say $A$ is {\sf generated in degree one} if $A$ is generated by $A_1$ as an $\kk$-algebra. 

While every Poisson filtration on a Poisson algebra $A$ naturally restricts to a filtration on the underlying associative algebra, not every algebra filtration $F$ on a Poisson algebra extends to a Poisson filtration.

\section{Cancellation properties for Poisson algebras}
\label{sec.zariski}
The main goal of our paper is to study the Zariski cancellation problem for Poisson algebras (abbreviated as \textbf{PZCP}). It can be thought of as a natural extension 
of the Zariski cancellation problem for (noncommutative) algebras (abbreviated as \textbf{ZCP}) \cite{BZ2,LWZ1,LWZ2}. We regard all isomorphisms as isomorphisms of Poisson algebras unless otherwise noted.

\begin{definition}\label{def.can}
A Poisson algebra $A$ is {\sf universally Poisson cancellative} if $A \tensor R \iso B \tensor R$ implies $A \iso B$ for every Poisson algebra $B$ and every finitely generated commutative domain $R$ over $\kk$ with trivial Poisson bracket such that the natural map $\kk \rightarrow R \rightarrow R/I$ is an isomorphism for some Poisson ideal $I \subset R$.
In the special case that $R=\kk[t]$ (resp. $\kk[t_1,\hdots,t_n]$ for all $n \geq 1$), we say $A$ is {\sf Poisson cancellative} (resp. {\sf strongly Poisson cancellative}).
\end{definition}

Our first concern is to construct examples of non-cancellative Poisson algebras with nontrivial Poisson bracket. The following lemma provides us an easy way of producing non-cancellative Poisson algebras from non-cancellative commutative algebras. 

\begin{lemma}\label{lemm.noncan}
Let $A$ be a commutative algebra that is not (strongly/universally) cancellative. Let $B$ be a Poisson algebra with trivial Poisson center. 
Then $A\otimes B$ is not (strongly/universally) Poisson cancellative. 
\end{lemma}
\begin{proof}
We only prove the result for the universally Poisson cancellative case. The cancellative and strongly cancellative cases can be proved similarly. Since $A$ is not universally cancellative, there exist another commutative algebra $C \niso A$ and an affine commutative domain $R$ together with an ideal $I\subset R$ such that $\kk\to R\to R/I$ is an isomorphism such that $A\otimes R\cong C\otimes R$.
As a consequence, $(A\otimes B)\otimes R\cong (C\otimes B)\otimes R$. If $A\otimes B$ were universally Poisson cancellative, it would imply that $A\otimes B\cong B\otimes C$ and hence 
\[ A\cong A\otimes \kk\cong  A\otimes \pcnt(B)\cong \pcnt(A\otimes B)\cong \pcnt(C\otimes B)\cong C\otimes \pcnt(B)\cong C\otimes \kk \cong C,\]
a contradiction. So $A\otimes B$ is not universally Poisson cancellative.
\end{proof}

\begin{example}\label{nonZCP}
The following examples of commutative algebras are known to be non-cancellative.
\begin{enumerate}
    \item Hochster showed that the coordinate ring of the real sphere $\RR[x,y,z]/(x^2+y^2+z^2-1)$ is not cancellative \cite{Ho}. 
    \item Let $n\ge 1$ and let $B_n$ be the coordinate ring of the complex surface $x^ny=z^2-1$. Danielewski proved that $B_n$ is not cancellative. More precisely, $B_i[t]\cong B_j[s]$ but $B_i\not \cong B_j$ for all $i\neq j\ge 1$ \cite{Da}.
    \item Gupta showed that the polynomial ring $\kk[x_1,\dots,x_n]$ is not cancellative whenever $n\ge 3$ and $\text{char}(\kk)>0$ \cite{gupta1,gupta2}.
\end{enumerate}
Now by Lemma \ref{lemm.noncan}, we can take tensor products of the non-cancellative commutative algebras listed above with Poisson algebras with trivial Poisson center (see Corollary \ref{cor.cancel1} and Example \ref{ex.cancel1}) to produce non-cancellative Poisson algebras.
\end{example}

The next lemma is useful in cancellation problems. The proofs are standard.
\begin{lemma}\label{lemm.invariso}
Let $A$ and $B$ be two Poisson algebras such that $A[t_1,\dots,t_n]\cong B[s_1,\dots,s_n]$ for some integer $n\ge 1$. Then the following hold.
\begin{enumerate}
    \item $A$ has trivial Poisson bracket if and only if $B$ does.
    \item $A$ is noetherian if and only if $B$ is.
    \item $\Kdim A=\Kdim B$.
    \item $A$ is artinian if and only if $B$ is.
    \item $A$ is a field if and only if $B$ is.
    \item $A$ is a finite direct sum of fields if and only if $B$ is.
    \item $A$ is local artinian if and only if $B$ is. 
\end{enumerate}
\end{lemma}

The notion of detectability was introduced in \cite[Definition 3.1]{LWZ1} to deal with the \textbf{ZCP} for noncommutative algebras. The notion of retractability is due to Abhyankar, Eakin, and Heinzer \cite[p. 311]{AEH}, called \emph{invariance} at the time, and was also meant to handle the \textbf{ZCP}. It was later called \emph{retractability} by Lezama, Wang, and Zhang in \cite[Definition 2.1]{LWZ1}.
We adapt both of these for Poisson algebras here.

\begin{definition}
Let $A$ be a Poisson algebra with Poisson center $\pcnt=\pcnt(A)$.
\begin{enumerate}
\item We say that $A$ is {\sf strongly Poisson detectable} (resp. {\sf strongly $\pcnt$-detectable}) if, for any Poisson algebra $B$ and any integer $n \geq 1$, a Poisson algebra isomorphism $\phi:A[t_1,\hdots,t_n]\xrightarrow{\sim} B[s_1,\hdots,s_n]$ implies that $B[s_1,\hdots,s_n] = B[\phi(t_1),\hdots,\phi(t_n)]$
(resp. $\pcnt(A)[s_1,\hdots,s_n] = \pcnt(B)[\phi(t_1),\hdots,\phi(t_n)]$), 
    
\item We say that $A$ is {\sf strongly Poisson retractable} (resp. {\sf strongly $\pcnt$-retractable}) if, for any Poisson algebra $B$ and integer $n \geq 1$, any Poisson algebra isomorphism $\phi:A[t_1,\hdots,t_n] \xrightarrow{\sim} B[s_1,\hdots,s_n]$ implies that $\phi(A)=B$ (resp. $\phi(\pcnt(A)) = \pcnt(B)$).
\end{enumerate}
If either holds only when $n=1$, then we say $A$ is simply {\sf Poisson detectable} (resp. {\sf$\pcnt$-detectable}) or {\sf Poisson retractable} (resp. {\sf $\pcnt$-retractable}).
\end{definition}

The condition $B[s_1,\hdots,s_n] = B[\phi(t_1),\hdots,\phi(t_n)]$ is equivalent to $s_i \in B[\phi(t_1),\hdots,\phi(t_n)]$ for all $i$.

The following observation is clear.

\begin{lemma}\label{lemm.retrdc}
Let $A$ be a Poisson algebra. 
\begin{enumerate}
\item If $A$ is (strongly) Poisson retractable, then $A$ is (strongly) Poisson detectable, (strongly) $\pcnt$-retractable, (strongly) $\pcnt$-detectable, and (strongly) Poisson cancellative.
\item If $A$ is (strongly) $\pcnt$-retractable, then $A$ is (strongly) $\pcnt$-detectable.
\item If $A$ is (strongly) $\pcnt$-detectable, then $A$ is (strongly) Poisson detectable. 
\item If $\pcnt$ is (strongly) retractable as an algebra, then $A$ is (strongly) $\pcnt$-retractable.
\item If $\pcnt$ is (strongly) detectable as an algebra, then $A$ is (strongly) $\pcnt$-detectable.
\end{enumerate}
\end{lemma}
\begin{proof}
Here we only show (3). All of the remaining cases are verified easily. Suppose that $\phi: A[t_1,\dots,t_n]\xrightarrow{\sim} B[s_1,\dots,s_n]$ is a Poisson isomorphism for some Poisson algebra $B$ and some $n\ge 1$. Then $\phi$ induces an isomorphism between their Poisson centers:
\[
\phi: \pcnt(A)[t_1,\dots,t_n]\cong \pcnt(A[t_1,\dots,t_n])\xrightarrow{\sim} \pcnt(B[s_1,\dots,s_n])\cong\pcnt(B)[s_1,\dots,s_n].
\]
Thus $s_1,\dots,s_n\in \pcnt(B)[\phi(t_1),\dots,\phi(t_n)]\subseteq B[\phi(t_1),\dots,\phi(t_n)]$ because $A$ is (strongly) $\pcnt$-detectable. Hence $A$ is (strongly) Poisson detectable. 
\end{proof}

\begin{lemma}\label{lemm.directprop}
The following properties of Poisson algebras are preserved under finite direct sums: 
\begin{enumerate}
    \item being (strongly) Poisson cancellative,
    \item being (strongly) Poisson retractable,
    \item being (strongly) $\pcnt$-retractable,
    \item being (strongly) Poisson dectectable, and
    \item being (strongly) $\pcnt$-detectable.
\end{enumerate}
\end{lemma}
\begin{proof}
We only prove (1) and (2)-(5) can be proved similarly. Let $A$ and $B$ be Poisson algebras that are (strongly) Poisson cancellative and let $\phi: (A\bigoplus B)[t_1,\dots,t_n]\xrightarrow{\sim} C[s_1,\dots,s_n]$ be an isomorphism for some Poisson algebra $C$ and some $n\ge 1$. Let $e_1,e_2$ be the two orthogonal idempotents corresponding to the decomposition $A\bigoplus B$. It is easy to check that $\phi(e_1)$ and $\phi(e_2)$ are two orthogonal idempotents of $C[s_1,\dots,s_n]$, which are indeed orthogonal idempotents of $C$. Write the corresponding decomposition of $C$ as $C_1\bigoplus C_2$. So, $\phi$ induces two isomorphisms $\phi_1: A[t_1,\dots,t_n]\xrightarrow{\sim} C_1[s_1,\dots,s_n]$ and $\phi_2:B[t_1,\dots,t_n]\xrightarrow{\sim} C_2[s_1,\dots,s_n]$. Since $A$ and $B$ are (strongly) Poisson cancellative, we have $A\cong C_1$, $B\cong C_2$, and $A\bigoplus B\cong C_1\bigoplus C_2\cong C$. Thus, $A\bigoplus B$ is (strongly) Poisson cancellative.
\end{proof}

\section{Graded versus ungraded}
\label{sec.graded}
In this section, we deal with the \textbf{PZCP} related to connected graded Poisson algebras following the ideas in \cite{BZ1}. To do this, we study the isomorphism problem for graded Poisson algebras. There has been significant progress in studying this problem in a variety of noncommutative situations \cite{BJ,BZ1,Giso,Giso2,GH,LY}. By \cite[Theorem 1]{BZ1}, an (ungraded) isomorphism between two connected graded algebras finitely generated in degree one implies the existence of a graded isomorphism. We prove a corresponding result in the Poisson setting. 

The following lemma is extracted from \cite[Lemma 1.1]{BZ1}. We keep its proof for the sake of completeness.

\begin{lemma}
\label{lem.igp}
Let $A$ and $B$ be two connected graded algebras that are generated in degree one. Suppose $\dim A_1<\infty$. If $\phi: A\to B$ is an isomorphism as ungraded algebras, then $\dim B_1=\dim A_1$.  
\end{lemma}
\begin{proof}
For any connected graded algebra $R$ with any codimension 1 ideal $I$, the tangent dimension of $I$ is said to be $\dim I/I^2$. It is easy to show that the tangent dimension of the augmentation ideal $R_{\ge 1}$ is the upper bound for that of any codimension 1 ideal of $R$. Now since $A$ is also generated by $A_1$, any codimension 1 ideal of $A$ has tangent dimension at most $\dim A_{\ge 1}/(A_{\ge 1})^2=\dim A_1$. The isomorphism $\phi: A\to B$ yields a bijection between their codimension 1 ideals sharing the same tangent dimensions. This implies that $\dim B_1=\dim B_{\ge 1}/(B_{\ge 1})^2=\dim A_1$. 
\end{proof}

The following result uses the proof of the isomorphism lemma \cite[Theorem 1]{BZ1} for noncommutative algebras. 

\begin{theorem}
\label{thm.iso}
Let $A$ and $B$ be two connected graded Poisson algebras finitely generated in degree one. If $A\cong B$ as ungraded Poisson algebras, then $A\cong B$ as graded Poisson algebras. 
\end{theorem}
\begin{proof}
Let $\phi: A\xrightarrow{\sim} B$ be such an ungraded isomorphism. From Lemma \ref{lem.igp}, $\dim A_1=\dim B_1=d$. Say $A_1=\text{span}(x_1,\dots,x_d)$ and $B_1=\text{span}(y_1,\dots,y_d)$. Note that $\phi(A_{\ge 1})$ is a codimension 1 Poisson ideal in $B$. By changing bases, without loss of generality, we can write
\begin{equation}
\label{eq.iso}
\phi(x_i)=y_i+y_i',\ 1\le i\le d-1,\ \text{and}\ \phi(x_d)=\alpha_d+y_d+y_d',
\end{equation}
where $y_1',\dots,y_d'\in B_{\ge 2}$ and $\alpha_d\in \kk$. Write $I=A_{\ge 1}$ and $J=B_{\ge 1}$. If $\alpha_d=0$, then $\phi(I)=J$. After taking the associated graded Poisson algebras, we have $A\cong \gr_I A\cong \gr_J B\cong B$ as graded Poisson algebras. 

Now we assume $\alpha_d\neq 0$. We first show the following two claims together by induction on the degree $N\ge 0$.

Claim I: If $r(x_1,\dots,x_d)$ is a homogeneous relation of degree $N$ in $A$, then $r(y_1,\dots,y_d)=0$ in $B$.

Claim II: $\dim A_N=\dim B_N$.

Claims I and II are trivial for $N=0,1$. Suppose they hold for $N\le m-1$. Now let $r(x_1,\dots,x_d)$ be a homogeneous relation of degree $m$ in $A$. Write 
$$r(x_1,\dots,x_d)=\sum_{0\le s\le m} g_s(x_1,\dots,x_{d-1})x_d^s,$$
where $g_s(x_1,\dots,x_{d-1})$ is a homogeneous term in $A$ of degree $m-s$. Let $s_0$ be the largest integer $s$ such that $g_s\neq 0$ in $A$. If $s_0=0$, then $r(x_1,\dots,x_d)=g_0(x_1,\dots,x_{d-1})$. By equation \eqref{eq.iso}, the lowest degree term of 
\begin{align*}
0   &= \phi(r(x_1,\dots,x_d))
    = r(y_1+y_1',\dots,y_{d-1}+y_{d-1}',\alpha_d+y_d+y_d') \\
    &= g_0(y_1+y_1',\dots,y_{d-1}+y_{d-1}')
\end{align*}
is $g_0(y_1,\dots,y_{d-1})$. So, $r(y_1,\dots,y_d)=g_0(y_1,\dots,y_{d-1})=0$ in $B$. If $s_0\ge 1$, write $r(x_1,\dots,x_d)=\sum_{0\le s\le s_0} g_s x_d^s$ with $g_{s_0}\neq 0$ and 
\[ 0 = \phi(r(x_1,\dots,x_d))
    =\sum_{0\le s\le s_0}g_s(y_1+y_1',\dots,y_{d-1}+y_{d-1}')(\alpha_d+y_d+y_d')^s,\]
whose lowest degree term is $g_{s_0}(y_1,\dots,y_{d-1})\alpha_d^{s_0}=0$. Since $\alpha_d\neq 0$, we get $g_{s_0}(y_1,\dots,y_{d-1})=0$ in $B$ but $g_{s_0}(x_1,\dots,x_{d-1})\neq 0$ in $A$ by the choice of $s_0$. By the induction hypotheses, Claims I and II in degree $m-s_0(<m)$ imply that the map $x_i\mapsto y_i$ for all $1\le i\le d$ induces an isomorphism between $A_{m-s_0}$ and $B_{m-s_0}$ as vector spaces. So, $g_{s_0}(y_1,\dots,y_{d-1})=0$ in $B_{m-s_0}$ implies that $g_{s_0}(x_1,\dots,x_{d-1})=0$ in $A_{m-s_0}$ as well. This contradicts our choice of $s_0$. Hence this case cannot happen and Claim I is true in degree $m$. Claim II in degree $m$ follows from Claim I since the map $x_i\mapsto y_i$ sends all relations of degree $m$ to $0$. Then $\dim A_m\ge \dim B_m$. By symmetry, the inverse isomorphism $\phi^{-1}: B\to A$ implies that $\dim B_m\ge \dim A_m$, so $\dim A_m=\dim B_m$. In conclusion, if a Poisson algebra isomorphism $\phi: A\to B$ satisfies Equation \eqref{eq.iso}, then the modified map $\widetilde{\phi}: A\to B$ via $x_i\mapsto y_i$ for $1\le i\le d$ is a well-defined algebra isomorphism. 

It remains to show that $\widetilde{\phi}(\{p,q\})=\{\widetilde{\phi}(p),\widetilde{\phi}(q)\}$ for all homogeneous elements $p,q\in A$. We will proceed by induction on the total degree $n=\text{deg}(p)+\text{deg}(q)$. If $n\le 1$, it is trivial. If $n=2$, without loss of generality, we can assume $p=x_i,q=x_j$. Because the Poisson bracket is homogeneous, we can write 
$$\{x_i,x_j\}=r(x_1,\dots,x_{d-1})+\left(\sum_{1\le k\le d-1} b_kx_k\right)x_d+cx_d^2$$
for some $b_k,c\in \kk$ and $r(x_1,\dots,x_{d-1})\in A_2$. Applying the Poisson isomorphism $\phi$ in \eqref{eq.iso} to the above equation, we get 
\begin{align*}
\{y_i+y_i',y_j+y_j'\}
    &= r(y_1+y_1',\dots,y_{d-1}+y_{d-1}')
        +\left(\sum_{1\le k\le d-1} (b_ky_k+b_ky_k')\right)(\alpha_d+y_d++y_d') \\ 
        &\qquad +c(\alpha_d+y_d+y_d')^2.
\end{align*}
Since the left hand side is of degree $\ge 2$, the degree 0 part of the right hand side yields that $c\alpha_d^2=0$. A similar argument suggests that we can assume $\alpha_d\neq 0$ and $c=0$. Then the degree 1 part of the right hand side gives that $\alpha_d(\sum_{1\le k\le d-1} b_ky_k)=0$. So we have all $b_k=0$ and $\{x_i,x_j\}=r(x_1,\dots,x_{d-1})$. Then by comparing the degree 2 parts in both sides, we get 
$$\widetilde{\phi}(\{x_i,x_j\})=r(y_1,\dots,y_{d-1})=\{y_i,y_j\}=\{\widetilde{\phi}(x_i),\widetilde{\phi}(x_j)\}.$$
Finally, suppose that $n\ge 2$. We can write $p=x_ib$ with $\text{deg}(b)+\text{deg}(q)<n$. Note that $\widetilde{\phi}$ is an algebra map and by induction hypothesis we have
\begin{align*}
\widetilde{\phi}(\{x_ib,q\})&\ =\widetilde{\phi}(x_i\{b,q\}+b\{x_i,q\})\\
&\ =\widetilde{\phi}(x_i)\widetilde{\phi}(\{b,q\})+\widetilde{\phi}(b)\widetilde{\phi}(\{x_i,q\})\\
&\ =\widetilde{\phi}(x_i)\{\widetilde{\phi}(b),\widetilde{\phi}(q)\}+\widetilde{\phi}(b)\{\widetilde{\phi}(x_i),\widetilde{\phi}(q)\}\\
&\ =\{\widetilde{\phi}(x_i)\widetilde{\phi}(b),\widetilde{\phi}(q)\}\\
&\ =\{\widetilde{\phi}(x_ib),\widetilde{\phi}(q)\}\\
&\ =\{\widetilde{\phi}(p),\widetilde{\phi}(q)\}.
\end{align*}
This completes our proof. 
\end{proof}

The following corollaries are immediate consequences of the isomorphism lemma for connected graded Poisson algebras. 
\begin{corollary}
\label{cor.poisdecom}
Suppose that a Poisson algebra $A$ has two graded Poisson algebra decompositions such that 
\[ A~=~\bigoplus_{i\ge 0} A_i~=~\bigoplus_{i\ge 0} A'_i\] where 
\begin{enumerate}
\item $A_0=A'_0=\kk$,
\item $A$ is generated by $A_1$ (respectively by $A_1'$) as an algebra, and
\item either $A_1$ or $A_1'$ is finite dimensional over $\kk$.
\end{enumerate}
Then there is a Poisson automorphism $\phi: A\to A$ such that $\phi(A_i)=A_i'$ for all integers $i\ge 0$.
\end{corollary}

We will use $\text{Aut}_P(A)$ (resp. $\text{Aut}_{grP}(A)$) to denote the group of all (graded) Poisson automorphisms of a (graded) Poisson algebra $A$.

\begin{corollary}
Retain the hypotheses of Corollary \ref{cor.poisdecom}. If $\text{Aut}_P(A)=\text{Aut}_{grP}(A)$, then $A_i=A_i'$ for all integers $i\ge 0$.
\end{corollary}

Now we can state the main result regarding the \textbf{PZCP} in the connected graded case with a restriction on the Poisson center being generated in degree at least $2$. It can be viewed as a Poisson version of \cite[Theorem 9]{BZ1}. 

\begin{theorem}
\label{thm.zariski1}
Let $A$ and $B$ be two connected graded Poisson algebras finitely generated in degree one. Suppose either $\pcnt(A)$ or $\pcnt(B)$ is generated in degree at least $2$. If $A[t_1,\dots,t_n]\cong B[s_1,\dots,s_n]$ as ungraded Poisson algebras, then $A\cong B$ as connected graded Poisson algebras.
\end{theorem}
\begin{proof}
If we set $\deg(t_i)=\deg(s_i)=1$ for all $i$, both $C:=A[t_1,\dots,t_n]$ and $D:=B[s_1,\dots,s_n]$ are connected graded Poisson algebras that are finitely generated in degree one. Since $C\cong D$, by Theorem \ref{thm.iso}, there is a graded Poisson algebra isomorphism $\phi: C\xrightarrow{\sim} D$. Without loss of generality, we can take $\pcnt(A)\cap A_1=\{0\}$ which implies $\pcnt(C)\cap C_1=\bigoplus_{1\le i\le n}\kk\, t_i$. The $s_j$'s are in the Poisson center of $D$. Therefore $\phi^{-1}(s_j)\in \bigoplus_{1\le i\le n}\kk\, t_i$ for all $j$. By a dimension argument, $\phi^{-1}(\bigoplus_{1\le i\le n}\kk\, s_i)=\bigoplus_{1\le i\le n}\kk\, t_i$. Modulo $s_i$ and $t_i$ we obtain an induced connected graded Poisson algebra isomorphism $\phi: A\cong C/(t_i)\cong D/(s_i)\cong B$.
\end{proof}

As an application, we can solve the isomorphism problem between skew quadratic Poisson algebras. It can be viewed as the semiclassical limit version of the isomorphism problem for skew polynomial algebras, which was studied in \cite{BZ1}.

\begin{theorem}
\label{thm.skewP}
Suppose $(\lambda_{ij})_{1\le i,j\le n}$ is some skew-symmetric matrix with $\lambda_{ij}\neq 0$ for all $i\neq j$. Let $A$ be the quotient of a skew quadratic Poisson algebra $\kk_{\lambda_{ij}}[x_1,\dots,x_n]/M$ where $\{x_i,x_j\}=\lambda_{ij}x_ix_j$ and $M$ is a graded Poisson ideal in $\kk_{\lambda_{ij}}[x_1,\dots,x_n]_{\ge 3}$. Let $B$ be another graded Poisson algebra $\kk_{\lambda'_{ij}}[x_1,\dots,x_m]/N$ where $N$ is a graded Poisson ideal in $\kk_{\lambda'_{ij}}[x_1,\dots,x_m]_{\ge 2}$. If $A$ is isomorphic to $B$ as ungraded Poisson algebras, then $n=m$ and there is a permutation $\sigma\in S_n$ such that $\lambda_{ij}=\lambda'_{\sigma(i)\sigma(j)}$ for all $i,j$. Furthermore, after an elementary change of generators in $A$, $A=B$.  
\end{theorem}
\begin{proof}
First of all, we show under the restrictions on $\lambda_{ij}$ that for any $f\in A_1$, $(f)$ is a Poisson ideal if and only if $f=cx_i$ for some scalar $c\in \kk$. One direction is clear. Conversely, let $f=\sum a_ix_i\in A_1$ where the coefficients for at least two variables are not zero. We need to show that $(f)$ is not a Poisson ideal. Without loss of generality, we can take $a_1=a_2=1$. Suppose it is not true. Then for all $i$ we can write 
$$\{x_i,f\}=\left(\sum_{1\le j\le n} b_jx_j\right)f$$
for some scalars $b_j\in\kk$. Taking the above identity in the quotient Poisson algebra $A'=A/(x_i)$ we get $0=(\sum_{j\neq i} b_jx_j)f$. Since the relation ideal $M$ of $A$ is contained in $\kk_{\lambda_{ij}}[x_1,\dots,x_n]_{\ge 3}$ and $a_1=a_2=1$ in $f$, then $f\neq 0$ in $A'$ and $b_j=0$ for all $j\neq i$. Back in $A$, 
$\{x_i,f\}=b_ix_if$ and 
\[ x_i\left(\sum_{1\le j\le n} (\lambda_{ij}-b_i)a_jx_j\right)=0.\]
The same argument as before shows that $(\lambda_{ij}-b_i)a_j=0$ for all $i,j$. Let $j=1,2$. Then $\lambda_{i1}=\lambda_{i2}=b_i$ for all $i$. In particular, $\lambda_{12}=\lambda_{11}=0$, which contradicts the assumptions on $\lambda_{ij}$. 

Next, by Theorem \ref{thm.iso}, $A$ is isomorphic to $B$ as graded Poisson algebras. In particular, $n=m$. Let $\phi: B\to A$ be such a graded Poisson isomorphism. Since the $(\phi(x_i))$ are Poisson ideals in $A$ generated by one degree one element, by the above discussion, $\phi(x_i)=c_ix_{\sigma(i)}$ for some $c_i\in\kk$ and some permutation $\sigma\in S_n$. Up to an elementary change of basis of $A$, $\phi$ sends $x_i$ to $x_i$ for all $i$. The result follows. 
\end{proof}

\section{Artinian center and detectability}
\label{sec.artcnt}
In this section, we prove that if a Poisson algebra $A$ has artinian Poisson center, then $A$ is strongly Poisson cancellative. As an application, we will show Poisson domains of Krull dimension two with nontrivial Poisson bracket are universally Poisson cancellative. 

\begin{lemma}\label{lemm.dectext}
Let $A$ be a Poisson algebra over $\kk$ and let $\kk'/\kk$ be any field extension. 
\begin{enumerate}
    \item If $A\otimes_\kk \kk'$ is (strongly) Poisson detectable, then so is $A$.
    \item If $A\otimes_\kk \kk'$ is (strongly) $\pcnt$-detectable, then so is $A$.
\end{enumerate}
\end{lemma}
\begin{proof}
(1) Let $\phi: A[t_1,\dots,t_n]\xrightarrow{\sim} B[s_1,\dots,s_n]$ be an isomorphism for some Poisson algebra $B$ over $\kk$ and some $n\ge 1$. Write $A'=A\otimes_\kk \kk'$ and $B'=B\otimes_\kk \kk'$. Then $\phi$ induces an isomorphism $\phi': A'[t_1,\dots,t_n]\cong B'[s_1,\dots,s_n]$. Since $A'$ is (strongly) Poisson detectable, we have 
\[
(B[\phi(t_1),\dots,\phi(t_n)])\otimes_\kk \kk'\cong B'[\phi'(t_1),\dots,\phi'(t_n)]=B'[s_1,\dots,s_n]\cong (B[s_1,\dots,s_n])\otimes_\kk \kk'.
\]
Because $\kk'/\kk$ is flat, we get $B[\phi(t_1),\dots,\phi(t_n)]=B[s_1,\dots,s_n]$ and $A$ is (strongly) Poisson detectable.

(2) can be proved similarly by noting that $\pcnt(A\otimes_\kk \kk')=\pcnt(A)\otimes_\kk \kk'$.
\end{proof}

A Poisson ideal that is also a prime ideal is called {\sf Poisson prime}. Recall that the nilradical $N(A)$ of a commutative algebra $A$ is the intersection of all its prime ideals which consists of all nilpotent elements. In particular, for any noetherian Poisson algebra $A$, \cite[Lemma 1.1(d)]{GoodPoisson} shows that any prime ideal of $A$ contains a Poisson prime. This implies that the nilradical of $A$ is the intersection of all Poisson prime ideals, which turns out to be a Poisson ideal. We denote the {\sf reduced Poisson algebra} of $A$ by  $A_{red}=A/N(A)$.

\begin{lemma}\label{lemm.nilrad}
Let $A$ be a noetherian Poisson algebra. If $A/N(A)$ is (strongly) Poisson detectable, so is $A$.
\end{lemma}
\begin{proof}
Suppose there is an isomorphism $\phi: A[t_1,\dots,t_n]\xrightarrow{\sim} B[s_1,\dots,s_n]$ for some Poisson algebra $B$ and some $n\ge 1$. 
Since $N(A[t_1,\dots,t_n]) = N(A)[t_1,\hdots,t_n]$ and similarly for $B$, then $\phi$ induces an isomorphism $\phi': (A/N(A))[t_1,\dots,t_n]\cong (B/N(B))[s_1,\dots,s_n]$. As $A/N(A)$ is (strongly) Poisson detectable, then $(B/N(B))[s_1,\dots,s_n]=(B/N(B))[\phi'(s_1),\dots,\phi'(s_n)]$. In another way, we have 
\[B[s_1,\dots,s_n]=B[\phi(s_1),\dots,\phi(s_n)]+N(B)[s_1,\dots,s_n].\]
Repeating this we get 
\[B[s_1,\dots,s_n]=B[\phi(s_1),\dots,\phi(s_n)]+N(B)^m[s_1,\dots,s_n],\ \text{for any}\ m\ge 1.\]
By Lemma \ref{lemm.invariso}(2), $B$ is noetherian and its nilradical $N(B)$ is nilpotent. Therefore, we get $B[s_1,\dots,s_n]=B[\phi(s_1),\dots,\phi(s_n)]$ and hence $A$ is (strongly) Poisson detectable.
\end{proof}

\begin{theorem}\label{thm.dectcan}
Let $A$ be a noetherian Poisson algebra. 
\begin{enumerate}
  \item If $A$ or $A_{red}$ is (strongly) Poisson detectable, then $A$ is (strongly) Poisson cancellative.
    \item If $A$ is (strongly) $\pcnt$-detectable, then $A$ is (strongly) Poisson cancellative. 
    \item If $A$ is (strongly) $\pcnt$-retractable, then $A$ is (strongly) Poisson cancellative. 
\end{enumerate}
\end{theorem}
\begin{proof}
(1) By Lemma \ref{lemm.nilrad}, $A$ is always (strongly) Poisson detectable. Let $\phi: A[t_1,\dots,t_n]\xrightarrow{\sim} B[s_1,\dots,s_n]$ be an isomorphism for some Poisson algebra $B$ and some $n\ge 1$. Since $\phi$ sends $t_1,\dots,t_n$ to $\pcnt(B)[s_1,\dots,s_n]$, there is a natural map 
$$f: B[s_1,\dots,s_n]\twoheadrightarrow B[\phi(t_1),\dots,\phi(t_n)]\subseteq B[s_1,\dots,s_n],\ \text{via}\ f(s_i)=\phi(t_i)\ \text{for all}\ 1\le i\le n.$$
By the fact that $A$ is (strongly) detectable, $B[\phi(t_1),\dots,\phi(t_n)]=B[s_1,\dots,s_n]$. Then $f$ is a surjective endomorphism of $B[s_1,\dots,s_n]$. By Lemma \ref{lemm.invariso}(2), $B$ is noetherian and so is $B[s_1,\dots,s_n]$. Since noetherian rings are Hopfian, $f$ is an automorphism by \cite[Proposition(i)]{Hir}. Therefore
\begin{align*}
    A\cong A[t_1,\dots,t_n]/(t_1,\dots,t_n) 
    &\xrightarrow{\phi} B[s_1,\dots,s_n]/(\phi(t_1),\dots,\phi(t_n)) \\ 
    &\xrightarrow{f}B[\phi(t_1),\dots,\phi(t_n)]/(\phi(t_1),\dots,\phi(t_n))\cong B.
\end{align*}
So, $A$ is (strongly) Poisson cancellative.

(2) is from (1) and Lemma \ref{lemm.retrdc}(3). 

(3) is from (3) and Lemma \ref{lemm.retrdc}(2).
\end{proof}

\begin{corollary}\label{lemm.artincan}
Let $A$ be a noetherian Poisson algebra. If $A$ or $A_{red}$ has artinian Poisson center, then $A$ is strongly Poisson cancellative. As a consequence, any artinian Poisson algebra is strongly Poisson cancellative. 
\end{corollary}
\begin{proof}
Suppose $A$ has artinian Poisson center $\pcnt$ with nilradical $N$. In the artinian case, $N$ is just the Jacobson radical of $\pcnt$. By a possible base field extension, we can assume that $\pcnt/N$ is finite direct sum of $\kk$. By Lemma \ref{lemm.dectext} and Lemma \ref{lemm.directprop}(4), $\pcnt/N$ is (strongly) detectable. Hence $\pcnt$ is (strongly) detectable by Lemma \ref{lemm.nilrad}. Then by Lemma \ref{lemm.retrdc}(5) $A$ is (strongly) $\pcnt$-detectable. Thus $A$ is (strongly) Poisson cancellative by Theorem \ref{thm.dectcan}(3). Similarly, if $A_{red}$ has artinian Poisson center, we can prove that $A_{red}$ is strongly Poisson detectable. So the result follows by Theorem \ref{thm.dectcan}(1). 

The same argument applies to any artinian Poisson algebra. 
\end{proof}

When the Poisson center of a Poisson algebra is trivial, we can push our result a little bit further in terms of being universally Poisson cancellative. This result and its proof are Poisson analogues of \cite[Proposition 1.3]{BZ2}.

\begin{theorem}
\label{thm.univ}
Let $A$ be a Poisson algebra with trivial Poisson center.
Then $A$ is universally Poisson cancellative.
\end{theorem}
\begin{proof}
Let $R$ be an affine commutative domain equipped with a trivial Poisson bracket. Assume $R/I = \kk$ for some (Poisson) ideal $I \subset R$ and suppose $\phi:A \tensor R \xrightarrow{\sim} B \tensor R$ is a Poisson algebra isomorphism for some Poisson algebra $B$.
Because $\pcnt(A)=\kk$, then $\pcnt(A \tensor R) = R$ and $\pcnt(B \tensor R) = \pcnt(B) \tensor R$. Since $\phi$ restricts to an isomorphism of the Poisson centers, we have $R \iso \pcnt(B) \tensor R$. Thus, $\pcnt(B)$ is a domain with $\Kdim(\pcnt(B))=0$. It follows that $\pcnt(B)$ is a field and since $I$ is a Poisson ideal such that $R/I = \kk$, then $\pcnt(B)=\kk$ and so $\pcnt(B \tensor R) = R$. Then the induced isomorphism $R=\pcnt(A \tensor R) \xrightarrow{\sim} \pcnt(B \tensor R)=R$ implies that $R/\phi(I)=\kk$. So $A \iso A \tensor (R/I) \iso B \tensor (R/\phi(I)) \iso B$.
\end{proof}

\begin{corollary}
\label{cor.cancel1}
Let $\text{char}(\kk)=0$ and let $A$ be an affine Poisson domain of Krull dimension two. If $A$ has nontrivial Poisson bracket, then $A$ is strongly Poisson cancellative. Moreover, if $\kk$ is algebraically closed, then $A$ is universally Poisson cancellative.
\end{corollary}

\begin{proof}
We first show that the Poisson bracket of any affine Poisson integral domain $B$ of Krull dimension at most one has to be trivial. Take $D$ to be the quotient field of $B$, which is a Poisson field after extending the Poisson bracket of $B$ to $D$. Therefore, $D$ has transcendence degree at most one by \cite[Chapter 5, \S14]{MA}. We treat the case when $D=\kk(z)(x_1,\dots,x_n)$ has transcendence degree one. The transcendence degree zero case can be proved similarly. Suppose $\{z,x_i\}=d\neq 0$ for some $x_i$. Consider the minimal polynomial of $x_i$ over $\kk(z)$ as $x^m+a_{m-1}x^{m-1}+\cdots+ a_0=0$ for some $m\ge 1$ and $a_j\in \kk(z)$. Applying the derivation $\{z,-\}$ to the above minimal polynomial and noticing that $\{z,a_j\}=0$, we get
\[(x^{m-1}+(m-1)a_{m-1}/ m\,x^{m-2}+\cdots+a_1/m)d=0.\] 
Since $d\neq 0$, $x^{m-1}+(m-1)a_{m-1}/ m\,x^{m-2}+\cdots+a_1/m=0$, contradicting the minimality of $m$. So, $\{z,x_1\}=\cdots=\{z,x_n\}=0$ and $z\in \pcnt(D)$. It remains to show that $x_1,\dots,x_n\in \pcnt(D)$. 
If $\{y,x_i\}\neq 0$ for some $y\in D$ and $x_i$ we still apply the derivation $\{y,-\}$ to the minimal polynomial of $x_i$ over $k(z)$. Since $\{y,z\}=0$, we get a similar contradiction of the minimal degree as before. So, $D$ has trivial Poisson bracket and hence $B$ has trivial Poisson bracket.   

Now let $Z=\pcnt(A)$. Denote by $A_Z$ the localization of $A$ at $Z$ and by $F$ the quotient field of $Z$. Since $A$ is an integral domain, then $A_Z \neq 0$. As a consequence, $\GKdim(A) \geq \GKdim_F(A_Z) + \GKdim(Z)$ \cite[Corollary 2]{SZ}. As $A$ is affine over $\kk$, one sees that $A_Z$ is affine over $F$. Hence $\GKdim(A)=\Kdim(A)=2$ and $\GKdim_F(A_Z)=\Kdim(A_Z)$ by \cite[Theorem 4.5]{GL}. By previous discussion, one sees that $\Kdim(A_Z)\ge 2$ since it has nontrivial Poisson bracket. Hence, $\GKdim(Z)=0$. So $\pcnt(A)=Z$ is an algebraic field extension of $\kk$ and is artinian. Hence $A$ is strongly Poisson cancellative by Corollary \ref{lemm.artincan}(1). 

In particular, if $\kk$ is algebraically closed the above argument shows that $\pcnt(A)=\kk$. We now can apply Theorem \ref{thm.univ} to conclude the proof.
\end{proof}

\begin{example}
\label{ex.cancel1}
As applications of Corollary \ref{lemm.artincan} and Theorem \ref{thm.univ}, we have the following examples. 
\begin{enumerate}
\item Let $\text{char}(\kk)=0$. The {\sf $n$th Poisson symplectic algebra} is $P_n=\kk[x_1,\dots,x_n,y_1,\dots,y_n]$
with the Poisson bracket
$$\{x_i,y_j\}=\delta_{ij},\ \{x_i,x_j\}=0,\ \{y_i,y_j\}=0.$$
A straightforward check shows $\pcnt(P_n)=\kk$ so $P_n$ is universally Poisson cancellative. 
\item Let $\text{char}(\kk)=0$ and $q\in \kk^\times$. Let $A=\kk[x_1,\dots,x_n]$ be the quadratic Poisson polynomial algebra with the Poisson bracket given by $\{x_i,x_j\}=qx_ix_j$ for all $1\le i<j\le n$. If $n\ge 2$, then $\pcnt(A)=\kk$ and $A$ is universally Poisson cancellative.
\item Let $\text{char}(\kk)=p>0$ and $\mathfrak g$ be a finite-dimensional restricted Lie algebra over $\kk$. Write $\mathfrak g=\text{span}(x_1,\dots,x_n)$ and $A=\kk[x_1,\dots,x_n]/(x_1^p,\dots,x_n^p)$. Then $A$ is a finite-dimensional Poisson algebra with the Kirillov-Kostant Poisson bracket such that $\{x_i,x_j\}=[x_i,x_j]$ for all $1\le i,j\le n$. So $A$ is strongly Poisson cancellative. 
\end{enumerate}
\end{example}

\section{The Makar-Limanov invariant and retractability}
\label{sec.ML}
In order to study the retractable property of Poisson algebras, we introduce an analogue of the Makar-Limanov invariant related to (higher) locally nilpotent Poisson derivations.

\begin{definition}
Let $A$ be a Poisson algebra with Poisson center $\pcnt$, let $0\neq d\in \pcnt$, and let $*$ be either blank or $\H$.
\begin{enumerate}
   \item For a higher Poisson derivation $\partial=(\partial_i)_{i=0}^\infty$, the kernel of $\partial$ is defined to be 
    \[ \ker \partial=\bigcap_{i\ge 1} \ker \partial_i.\]
  \item  The set of all {\sf (higher) locally nilpotent Poisson derivations} of $A$ with respect to $d$ is denoted 
    \[\plndds(A) = \left\{ \partial \in \plnds(A) : d\in \ker \partial\right\}.\]
\item The {\sf Poisson Makar-Limanov\,$\!^{*}_{d}$ invariant} of $A$ is defined to be
\[ \pmlds(A) = \bigcap_{\partial \in \plndds(A)} \ker(\partial).\]
     \item We say $A$ is {\sf $\plndds$-rigid} if $A=\pmlds(A)$, or equivalently, $\plndds(A)=\{0\}$. Moreover, we say $A$ is {\sf strongly $\plndds$-rigid} if $A\subseteq \pmlds(A[t_1,\dots,t_n])$, for all $n\ge 1$. 
     \item We say $A$ is {\sf $\plnddps$-rigid} if $\pcnt\subseteq\pmlds(A)$. Moreover, we say $A$ is {\sf strongly $\plnddps$-rigid} if $\pcnt\subseteq \pmlds(A[t_1,\dots,t_n])$ for all $n\ge 1$.
\end{enumerate}
\end{definition}

\begin{remark}
\begin{enumerate}
\item Throughout, $*$ will be either blank or $\H$. When $*$ is blank, we will further assume that $\text{char}(\kk)=0$. Note that there is an embedding from $\plndd(A)$ to $\plnddh(A)$ given by $\partial\mapsto (\partial^i/i!)_{i=0}^\infty$ when $\text{char}(\kk)=0$.
\item Note that $1\in \ker \partial$ for any $\partial\in \plnds(A)$. Hence when $d=1$ we will simply omit it and write $\plnds$, $\pmls$, (strongly) $\plnds$-rigid, or (strongly) $\plndps$-rigid. 
    \item When $A$ is a commutative algebra with trivial Poisson bracket, then $\pcnt=A$ and we simply write $\lndds$, $\mlds$, or (strongly) $\lndds$-rigid instead of $\plndds$, $\pmlds$, (strongly) $\plndds$ or, $\plnddps$-rigid, respectively. 
\end{enumerate}
\end{remark}

\begin{example}
\label{ex.plnd}
Let $A=\kk[x,y]$ and let $\delta=y \frac{\partial}{\partial x}$.
Clearly, $\delta \in \lnd(A)$. Now consider the Poisson bracket on $A$ given by $\{x,y\}=qxy$ for some $q\in \kk^\times$. One can check that $\delta\notin \pder(A)$. Hence, $\delta \notin \plnd(A)$. Furthermore, one can show that when $\text{char}(\kk)=0$ no nontrivial locally nilpotent derivation of $A$ is a Poisson derivation. It follows that $\plnd(A)=A$ and so $A$ is $\plnd$-rigid.
\end{example}

\begin{lemma}\label{lem:iterative}
Let $\partial:=\{\partial_i\}_{i=0}^\infty$ be a higher Poisson derivation of a Poisson algebra $A$.
\begin{enumerate}
\item Suppose $\partial$ is locally nilpotent. For every $c\in \kk$
there is a Poisson algebra automorphism of $A$, $G_{c\partial}: A\to A$, defined by $a\mapsto \sum_{i=0}^\infty c^i\partial_i(a)$.
\item If $\partial$ is iterative and if for any $a\in A$ there exists $n\ge 0$ such that $\partial_i(a)=0$ for all $i\ge n$, then $G_{\partial,t}: A[t]\to A[t]$ defined by $a\mapsto \sum_{i=0}^\infty \partial_i(a)t^i$ and $t\mapsto t$ is a Poisson algebra automorphism. As a a consequence, $\partial$ is locally nilpotent.
\item Let $G: A[t]\to A[t]$ be a Poisson $\kk[t]$-algebra automorphism. If $G(a)\equiv a \mod(t)$ for all $a\in A$, then $G=G_{\partial,t}$ for some $\partial\in \plndh(A)$.
\item If $\partial\in \plnddh(A)$ for some $0\neq d\in \pcnt(A)$, then $G_{c\partial}(d)=G_{\partial,t}(d)=d$.
\end{enumerate}
\end{lemma}
\begin{proof}
(1) By definition, $G_{\partial,t}$ is a Poisson algebra automorphism of $A[t]$ and $G_{\partial,t}(t-c)=t-c$. Note that $(t-c)$ is a Poisson ideal of $A[t]$. Then it yields a Poisson algebra automorphism of $A[t]/(t-c)$. One sees that the induced automorphism is exactly $G_{c\partial}$. 

(2) First we show that $G_{\partial,t}$ is an algebra automorphism with inverse given by $G_{\partial,-t}$. Since $G_{\partial,t}$ is $\kk[t]$-linear, it suffices to show that, for all $a,b\in A$, 
\begin{align*}
G_{\partial,t}(ab)
    &= \sum_{i=0}^\infty \partial_i(ab)t^i
    = \sum_{i=0}^\infty t^i\left(\sum_{j=0}^i \partial_j(a)\partial_{i-j}(b)\right) \\
    &= \sum_{j=0}^\infty \partial_j(a)t^j\left(\sum_{i=j}^\infty  \partial_{i-j}(b)t^{i-j}\right)
    = G_{\partial,t}(a)G_{\partial,t}(b),
\end{align*}
where all interchanging of summations can be justified by the fact that the sums are actually finite. Moreover, we have
\begin{align*}
G_{\partial,t}\circ G_{\partial,-t}(a)&=G_{\partial,t}\left(\sum_{i=0}^\infty \partial_i(a)(-t)^i\right)=\sum_{i=0}^\infty G_{\partial,t}(\partial_i(a))(-t)^i=\sum_{i=0}^\infty \left(\sum_{j=0}^\infty \partial_j\partial_i(a)t^j\right)(-t)^i\\
&=\sum_{i=0}^\infty \left(\sum_{j=0}^\infty \binom{i+j}{i} \partial_{i+j}(a)t^j\right)(-t)^i=\sum_{i=0}^\infty\sum_{n=i}^\infty \binom{n}{i}(-1)^i \partial_n(a)t^n\\
&=\sum_{n=0}^\infty \partial_n(a)t^n\left(\sum_{i=0}^n (-1)^i \binom{n}{i}\right)=\partial_0(a)=a.
\end{align*}

It remains to show that $G_{\partial,t}$ is a Poisson algebra homomorphism, which follows from, for any $a,b\in A$,
\begin{align*}
G_{\partial,t}(\{a,b\})
&=\sum_{i=0}^\infty \partial_i(\{a,b\})t^i
=\sum_{i=0}^\infty \sum_{j=0}^i \{\partial_j(a),\partial_{i-j}(b)\}t^i
=\sum_{i=0}^\infty \sum_{j=0}^i \{\partial_j(a)t^j,\partial_{i-j}(b)t^{i-j}\} \\
&=\left\{\sum_{i=0}^\infty \partial_i(a)t^i,\sum_{j=0}^\infty \partial_j(b)t^j\right\}
=\{G_{\partial,t}(a),G_{\partial,t}(b)\}.
\end{align*}

(3) Write $G(a)=\sum_{i\ge 0} \partial_i(a)t^i$ for all $a\in A$. In a similar way to the proof of (2), we have that $\partial:=\{\partial_i\}_{i=0}^\infty$ is in $\plndh(A)$.

(4) is clear from the definitions of $G_{c\partial}$ and $G_{\partial,t}$.
\end{proof}

\begin{lemma}
\label{lem.ddx}
Let $A$ be a Poisson algebra, and  $0\neq d\in \pcnt(A)$. 
Then $\pmlds(A[t_1,\dots,t_n])\subseteq A$ for any $n\ge 1$.
\end{lemma}
\begin{proof}
First suppose that $*$ is blank. That $\partial_i=\partial/\partial t_i$ is a $\kk$-derivation of $A[t_1,\hdots,t_n]$ such that $\partial_i(d)=0$ is clear. We claim it is a Poisson derivation.
By induction, it suffices to show that $\delta=\frac{d}{dt}$ is a Poisson derivation of $A[t]$.

Let $p = \sum_{i=0}^n a_i t^i$ and $q = \sum_{j=0}^n b_j t^j$ be in $A[t]$. Then
\begin{align*}
\delta( \{p,q\} )
	&= \sum_{i=0}^n \sum_{j=0}^n \delta( \{ a_i t^i, b_j t^j \} )
	= \sum_{i=0}^n \sum_{j=0}^n \delta( \{ a_i, b_j \} t^{i+j})
	= \sum_{i=0}^n \sum_{j=0}^n (i+j) \{ a_i, b_j \} t^{i+j-1} \\
	&= \sum_{i=1}^n \sum_{j=0}^n i \{ a_i, b_j \} t^{i+j-1}
		+ \sum_{i=0}^n \sum_{j=1}^n j \{ a_i, b_j \} t^{i+j-1} \\
	&= \left\lbrace \sum_{i=1}^n i a_i t^{i-1},\, q \right\rbrace 
		+ \left\lbrace p,\, \sum_{j=1}^n j b_j t^{j-1} \right\rbrace
	= \{ \delta(p), q \} + \{ p, \delta(q) \}.
\end{align*}

Now suppose $*=\H$. For each $1\le i\le n$, define the series of $A$-linear operators $\{\Delta_i^m\}_{m=0}^\infty$, where
\[
\Delta_i^m: t_1^{p_1}\cdots t_n^{p_n}\mapsto 
\begin{cases}
\binom{p_i}{m}t_1^{p_1}\cdots t_i^{p_i-m}\cdots t_n^{p_n} & \text{if}\ p_i\ge m\\ 
0 & \text{otherwise}.
\end{cases}
\]
We show that $\{\Delta_i^m\}_{m=0}^\infty$ belongs to $\plnddh(A[t_1,\dots,t_n])$ for all $1\le i\le n$.
Again, it suffices to show this for $n=1$.
Set $\delta=(\delta_n)_{n=0}^\infty$ for $A[t]$, where $\delta_n(at^m)=\binom{m}{n}at^{m-n}$ for $m\ge n$ and $\delta_n(at^m)=0$ if $n>m$. Then $(\delta_n)_{n=0}^\infty$ is clearly a higher derivation of $A[t]$ and 
\begin{align*}
\delta_n\left(\{at^p, bt^q\}\right)&=\delta_n\left(\{a,b\}t^{p+q}\right)
=\binom{p+q}{n}\{a,b\} t^{p+q-n}
=\left(\sum_{i=0}^n \binom{p}{i}\binom{q}{n-i}\right) \{a,b\} t^{p+q-n}\\
&=\sum_{i=0}^n\left\{ \binom{p}{i}at^{p-i},\ \binom{q}{n-i}bt^{p-(n-i)}\right\}
=\sum_{i=0}^n\{\delta_i(at^p),\ \delta_{n-i}(bt^q)\}.
\end{align*}
Thus, $(\delta_n)_{n=0}^\infty$ is a higher Poisson derivation of $A[t]$.
Moreover, one can easily check that $\{\delta_n\}_{n=0}^\infty$ is iterative and $\delta_n(d)=0$ for all $n\ge 1$. Hence by Lemma \ref{lem:iterative}(2) $\delta=\{\delta_n\}_{n=0}^\infty$ belongs to $\plnddh(A)$.

Finally, in all cases one can check directly $\pmlds(A[t_1,\dots,t_n])\subseteq \bigcap_{1\le i\le n} \ker \delta_i \subseteq A$.
\end{proof}

\begin{lemma}
\label{lem.subalg}
Let $Y= \bigoplus_{i=0}^\infty Y_i$ be an $\NN$-graded affine commutative domain.
If $Z$ is an affine subalgebra of $Y$ containing $Y_0$ such that
$\Kdim Z = \Kdim Y_0 < \infty$, then $Z=Y_0$.
\end{lemma}
\begin{proof}
This can be proved using a similar argument as in \cite[Lemma 3.2]{BZ2} after replacing $\GKdim$ with $\Kdim$.
\end{proof}

\begin{lemma}
\label{lemm.cancel}
Let $A$ be a Poisson algebra.
\begin{enumerate}
\item Suppose $A$ is an affine domain. If $\pmls(A[t])=A$, then $A$ is Poisson retractable.
\item Suppose $A$ is an affine domain. If $A$ is strongly $\plnds$-rigid, then $A$ is strongly Poisson retractable. 
\item Suppose $\pcnt(A)$ is an affine domain. If $\pcnt(A)\subseteq \pmls(A[t])$, then $A$ is $\pcnt$-retractable.
\item Suppose $\pcnt(A)$ is an affine domain. If  $A$ is strongly $\plndps$-rigid, then $A$ is strongly $\pcnt$-retractable.
\end{enumerate}
\end{lemma}
\begin{proof}
(2) Let $\phi:A[t_1,\hdots,t_n] \xrightarrow{\sim} B[s_1,\hdots,s_n]$ be an isomorphism for some Poisson algebra $B$ and some $n \geq 1$.
By Lemma \ref{lemm.invariso}(3), $\Kdim A =\Kdim B < \infty$. For any locally nilpotent (higher) Poisson derivation $\partial$ of $A[t_1,\dots,t_n]$, 
$\phi \circ \delta \circ \phi\inv$ is again a locally
nilpotent (higher) Poisson derivation of $B[s_1,\cdots,s_n]$.
Reversing this argument shows that $\phi$ induces an isomorphism
\[\phi: A=\pmls(A[t_1,\hdots,t_n]) \xrightarrow{\sim} \pmls(B[s_1,\hdots,s_n])\subseteq B.\]
The last inclusion is due to Lemma \ref{lem.ddx} since $A$ has finite Krull dimension. It follows that $\phi(A) \subseteq B$.
Let $Y=A[t_1,\hdots,t_n]$ with $\deg t_i=1$, $Y_0 = A$, and $Z=\phi\inv(B)$.
Then Lemma \ref{lem.subalg} implies that $\phi\inv(B)=A$.
Thus, $\phi(A)=B$ and $A$ is strongly Poisson retractable. 
The proof of (1) is analogous. 

(4) Let $\phi:A[t_1,\hdots,t_n] \xrightarrow{\sim} B[s_1,\hdots,s_n]$ be an isomorphism for some Poisson algebra $B$ and some $n \geq 1$. Since $\phi$ preserves the Poisson center, it induces an isomorphism 
\[\phi: \pmls(A[t_1,\hdots,t_n])\bigcap \pcnt(A)[t_1,\dots,t_n] \xrightarrow{\sim} \pmls(B[s_1,\hdots,s_n])\bigcap \pcnt(B)[s_1,\dots,s_n].\]
Note that $\pmls(B[s_1,\hdots,s_n])\bigcap \pcnt(B)[s_1,\dots,s_n]\subseteq \pcnt(B)$.
Because $A$ is strongly $\plndps$-rigid, we get $\pcnt(A)=\pmls(A[t_1,\hdots,t_n])\bigcap \pcnt(A)[t_1,\dots,t_n]$. Hence $\phi(\pcnt(A))\subseteq \pcnt(B)$. As in (2), we get $\phi(\pcnt(A))=\pcnt(B)$ and $A$ is strongly $\pcnt$-retractable.
The proof of (3) is analogous. 
\end{proof}

Our next main result shows that $\plnd$-rigidity implies Poisson cancellation, which suggests that the Poisson Makar-Limanov invariant is a useful invariant in the \textbf{PZCP}.

\begin{theorem}\label{thm.plndcan}
Let $A$ be a Poisson algebra. 
\begin{enumerate}
\item Suppose $A$ is an affine domain. If $\pmls(A[t])=A$, then $A$ is Poisson cancellative.
\item Suppose $A$ is an affine domain. If $A$ is strongly $\plnds$-rigid, then $A$ is strongly Poisson cancellative.
\item Suppose $A$ is noetherian and $\pcnt(A)$ is an affine domain. If $\pcnt(A)\subseteq \pmls(A[t])$, then $A$ is Poisson cancellative.
\item Suppose $A$ is noetherian and $\pcnt(A)$ is an affine domain. If $A$ is strongly $\plndps$-rigid, then $A$ is strongly Poisson cancellative.
\end{enumerate}
\end{theorem}
\begin{proof}
(1)-(2) are from Lemma \ref{lemm.cancel}(1)-(2) and Lemma \ref{lemm.retrdc}(1). 

(3)-(4) are from Lemma \ref{lemm.cancel}(3)-(4) and Theorem \ref{thm.dectcan}(3). 
\end{proof}

Our last result in this section is a Poisson analogue of \cite[Theorem 3.6]{BZ2}. We adapt the arguments in \cite[Lemma 3.4 and Lemma 3.5]{BZ2} and make the necessary changes for Poisson algebras. 

\begin{lemma}
\label{lem.embed}
Let $\text{char}(\kk)=0$ and let $A$ be an affine Poisson domain. Denote by $Q(A)$ the fractional quotient of $A$.
Suppose that $A$ is endowed with a nonzero locally nilpotent Poisson derivation $\alpha$. Then the following hold.
\begin{enumerate}
\item $A$ is embedded in the Poisson-Ore extension $E[x;0,\delta_0]_P$ and $E[x;0,\delta_0]_P$ 
is embedded in $Q(A)$, where $E= \{ a \in Q(A) : \alpha(a)=0 \}$
and $\delta_0$ is a Poisson derivation of $E$.
\item $Q(A) = Q(E[x;0,\delta_0]_P)$.
\item The Poisson derivation $\alpha$ can be extended to a locally nilpotent 
Poisson derivation of $E[x;0,\delta_0]_P$
by declaring that $\alpha(E)=0$ and $\alpha(x)=1$.
\end{enumerate}
\end{lemma}

\begin{proof}
(1) By \cite[Lemma 3.2]{OH1}, $\alpha$ extends uniquely to a Poisson derivation of $Q(A)$.
Let $E$ denote the kernel of this extension.
Because $\alpha$ is a Poisson derivation we have
\[ \alpha( \{ u, v \}) = \{ \alpha(u), v \} + \{ u, \alpha(v) \} = 0 \quad\text{for all $u,v \in E$}.\]
Hence, $E$ is a Poisson subalgebra of $Q(A)$.

By hypothesis $\alpha$ is nonzero and locally nilpotent and so there exists $x \in Q(A)\backslash E$ such that $\alpha(x) \in E$.
Moreover, we may choose $x$ such that $\alpha(x)=1$.
Thus, for all $a \in E$,
\[
\alpha( \{x,a\} )
	= \{ \alpha(x), a\}
	= \{ 1,a \}
	= 0.
\]
In particular, $\{x,-\}$ induces a derivation $\delta_0$ of $E$. By the Jacobi identity,
\begin{align*}
\delta_0( \{u,v\} )
	&= \{ x, \{u,v\} \}
	= -\{ u, \{v,x\} \} - \{ v, \{x,u\} \} \\
	&=  \{ u, \{x,v\} \} + \{ \{x,u\}, v \}
	= \{ u, \delta_0(v)\} + \{ \delta_0(u), v\}.
\end{align*}
Hence, $\{x,-\}$ induces a Poisson derivation $\delta_0$ of $E$.

Let $W = \{ a \in Q(A) : \alpha^n(a)=0 \text{ for some } n \geq 0 \}$.
Denote by $W'$ the subalgebra of $Q(A)$ generated by $E$ and $x$.
Since $\{ x, E \} \subseteq E$, then $W' = \sum_{i \geq 0} Ex^i$ is a Poisson subalgebra of $W$.
To prove $W = W'$, it suffices to show that $W \subseteq W'$.
Let $a \in W$ and let $n$ be the smallest integer such that $\alpha^n(a)=0$.
When $n=0$, $a \in E$ and the claim holds. We proceed inductively.
Assume the result holds for all $j < n$. Take $\alpha^n(a)=0$ and write $\alpha^{n-1}(a)=r \in E$ with $r\neq 0$ and $\alpha(r)=0$. But $\alpha^{n-1}(rx^{n-1}/(n-1)!) = r$ so 
$\alpha^{n-1}(a-rx^{n-1}/(n-1)!) = 0$. Now $rx^{n-1}/(n-1)! \in Ex^{n-1}$ and so $W=W'$.

As $\alpha \in \plnd(A)$, then $A \subset W$.
Thus, $A$ embeds in the Poisson subalgebra $W$.
Since $ \{x,a\}=\delta_0(a)$ for all $a \in E$, then
$W$ is isomorphic to a homomorphic image of the Poisson-Ore extension $E[x;0,\delta_0]_P$.
Let $I$ denote the kernel of this image and choose a monic
element of minimal $d$-degree, say $b = x^d + b_{d-1} x^{d-1} + \cdots + b_1 x + b_0 \in I$.
Applying $\alpha$ gives 
\[ x^{d-1} + \sum_{j=1}^{d-1} j d\inv b_j x^{j-1} \in I.\]
This contradicts the minimality of $d$ unless $d=0$.
Thus, $A$ embeds in $W \iso E[x;0,\delta_0]_P$. Finally, the remaining of the statements are easy to check by our construction. 
\end{proof}

\begin{lemma}
\label{lem.indPder}
Let $A$ be a Poisson algebra and let $\alpha$ be a Poisson derivation of the Poisson algebra $A[x]$.
Suppose $\alpha(A) \subset A + Ax + \cdots +A x^m$ for some fixed $m \geq 0$.
For $a \in A$, write $\alpha(a) = \sum_{i=0}^m c_{a,i} x^i$ for $c_{a,i}\in A$.
Then the map $\beta:A \rightarrow A$ defined by $\beta(a) = c_{a,m}$ defines a Poisson derivation of $A$. 
\end{lemma}
\begin{proof}
It is clear that $\beta$ is $\kk$-linear because $\alpha$ is and it is easy to check that $\beta$ is a derivation of $A$. We show that $\beta$ is a Poisson derivation. Let $a,b \in A$. We use the notation $\alpha(a)_m$ to denote the coefficient of $x^m$ in $\alpha(a)$. Then
\begin{align*}
\beta(\{a,b\})
	&= \alpha(\{a,b\})_m 
	= \left( \{ \alpha(a),b \} + \{a, \alpha(b)\} \right)_m \\
	&= \left( \left\lbrace \sum_{i=0}^m c_{a,i} x^i, b \right\rbrace
		+ \left\lbrace a, \sum_{i=0}^m c_{b,i} x^i \right\rbrace \right)_m
	= \left( \sum_{i=0}^m \left( \left\lbrace c_{a,i}, b \right\rbrace
		+ \left\lbrace a, c_{b,i} \right\rbrace \right) x^i \right)_m \\
	&= \left\lbrace c_{a,m}, b \right\rbrace + \left\lbrace a, c_{b,m} \right\rbrace
	= \{ \beta(a),b \} + \{a, \beta(b)\}. 
\end{align*}
\end{proof}

\begin{lemma}
\label{lem.plnd}
Let $\text{char}(\kk)=0$ and let $A$ be a Poisson algebra with some $0\neq d\in \pcnt(A)$.
\begin{enumerate}
    \item Suppose $A$ is an affine domain. If $A$ is $\plndd$-rigid, then $\pmld(A[x])=A$.
    \item Suppose $\pcnt(A)$ is an affine domain. If $A$ is $\plnddp$-rigid, then $\pcnt(A)\subseteq \pmld(A[x])$.
\end{enumerate}
\end{lemma}

\begin{proof}
(1) By Lemma \ref{lem.ddx}, it suffices to show $A \subseteq \pmld(A[x])$. Suppose by way of contradiction that there exists a locally nilpotent Poisson derivation $\alpha$ of $A[x]$ such that $\alpha(A)\neq 0$ and $\alpha(d)=0$.
Let $a_1,\hdots,a_s$ be a generating set of $A$ as a $\kk$-algebra.
Then $\alpha(A) \subset A\alpha(a_1) + \cdots + A\alpha(a_s)$ and so there exists $m \geq 0$ minimial such that $\alpha(A) \subset A + Ax + \cdots +A x^m$.
When $m=0$ we have $\alpha(A) \subset A$ and so $\alpha(A)=0$ because
$A$ is $\plndd$-rigid, a contradiction. Thus, $m \geq 1$. Write, by Lemma \ref{lem.indPder},
\[ \alpha(a) = \beta(a)x^m + \text{ lower degree terms }\]
for some $\beta \in \pder(A)$ and $\beta(d)=0$. 
We consider three cases depending on the image of $x$ under $\alpha$.
It will follow from these cases that $\alpha(A)=0$.

{\bf Case 1} $\alpha(x) \in A + Ax + \cdots +A x^m$.

In this case we have $\alpha(x^i)\subseteq\sum_{n=0}^{i+m-1}Ax^n$ and $\alpha(Ax^i)\subseteq \sum_{n=0}^{i+m}Ax^n$ for all $i$. Thus for every $a\in A$ we have
$$\alpha^j(a)=\beta^j(a)x^{mj}+ \text{ lower degree terms}.$$
We get $\beta$ is locally nilpotent as $\alpha$ is. Hence $\beta\in \plndd(A)$ by Lemma \ref{lem.indPder}, which implies that $\beta(A)=0$ for $A$ is $\plndd$-rigid. This contradicts the minimality of $m$ so $\alpha(A)=0$.

{\bf Case 2} $\alpha(x) = bx^{m+1} + \text{ lower degree terms}$ for some $b \neq 0$ in $A$.

Since $\{x,a\}=0$, then we have
\begin{align*}
0	&= \alpha( \{x,a\})
	= \{ \alpha(x), a \} + \{ x,\alpha(a) \} \\
	&= \{ bx^{m+1} + \text{lower degree terms}, a \} + \{ x, \beta(a)x^m + \text{lower degree terms}\} \\
	&= \{ b, a \} x^{m+1}+\text{lower degree terms}.
\end{align*}
Thus, $\{b,a\}=0$ for all $a \in A$ so $b \in \pcnt(A)$.
Define $\alpha':A[x] \rightarrow A[x]$ by $\alpha'(a) = \beta(a)x^m$ for $a \in A$ and $\alpha'(x)=bx^{m+1}$.
By \cite[Lemma 4.11]{CPWZ2}, $\alpha' \in \der(A[x])$. Moreover, we can check that $\alpha'$ is a Poisson derivation.

Let $f = \sum_{i=0}^n f_i x^i$ and $g = \sum_{j=0}^k g_j x^j$ be in $A[x]$. Then
\begin{align*}
\alpha'( \{ f, g\} )	
	&= \alpha'\left( \sum_{i=0}^n \sum_{j=0}^k \{ f_i, g_j \} x^{i+j} \right)
	= \sum_{i=0}^n \sum_{j=0}^k \left( \alpha'(\{ f_i, g_j \}) x^{i+j} 
		+ \{ f_i, g_j \} \alpha'(x^{i+j}) \right) \\
	&= \sum_{i=0}^n \sum_{j=0}^k \left( \beta(\{ f_i, g_j \}) x^{m+i+j} 
		+ \{ f_i, g_j \} ((i+j)bx^{m+i+j}) \right) \\
	&= \sum_{i=0}^n \sum_{j=0}^k \left(\{ \beta(f_i), g_j \} + \{ f_i, \beta(g_j) \}
		+ (i+j)b\{ f_i, g_j \} \right) x^{m+i+j} \\
	&= \sum_{i=0}^n \sum_{j=0}^k \left( \{ \beta(f_i) + f_i(ib), g_j \}
		+ \{f_i, \beta(g_j)+g_j(jb)\} \right) x^{m+i+j} \\
	&= \sum_{i=0}^n \sum_{j=0}^k \left( \{ \beta(f_i)x^{m+i} + f_i(ibx^{m+i}), g_j \} x^j 
		+ \{f_i, \beta(g_j)x^{m+j}+g_j(jbx^{m+j})\} x^i \right) \\
	&= \sum_{i=0}^n \sum_{j=0}^k \left( \{ \alpha'(f_i)x^i + f_i\alpha'(x^i), g_j \} x^j 
		+ \{f_i, \alpha'(g_j)x^j+g_j\alpha'(x^j)\} x^i \right) \\
	&= \{ \alpha'(f), g \} + \{f, \alpha'(g)\}.
\end{align*}
Thus $\alpha' \in \pder(A[x])$. We can view $\alpha'$ as an associated graded Poisson derivation of $\alpha$. Since $\alpha$ is locally nilpotent, so is $\alpha'$ and $\alpha'\in \plndd(A[x])$. We apply Lemma \ref{lem.embed} to the Poisson algebra $A[x]$ with respective to $\alpha'$. Denote $E=\{a\in Q(A[x]): \alpha'(a)=0\}$. There is an embedding $\phi: A[x]\hookrightarrow E[y;0,\delta_0]_P$ for some $\delta_0\in \pder(E)$. Moreover, $\alpha'$ extends to a locally nilpotent Poisson derivation of $E[y;0,\delta_0]_P$ by setting $\alpha'(E)=0$ and $\alpha'(y)=1$. Under this embedding, we have $\phi(bx^{m+1})=\phi(\alpha'(x))=\alpha'(\phi(x))$. It follows that 
\[ \deg_y \phi(b)+(m+1)\deg_y \phi(x)=\deg_y \phi(x)-1\ \text{in $E[y]$}.\]
Since $\alpha'(x)\neq 0$, $\deg_y\phi(x)\ge 1$, a contradiction to the above equality.

{\bf Case 3} $\alpha(x) = bx^i + \text{ lower degree terms}$ 
for some $b \neq 0$ in $A$ and some $i > m+1$. It is easy to see 
\[
\alpha^n(x)=\left\{\prod_{s=1}^{n-1}((i-1)s+1)\right\}b^nx^{(i-1)n+1}+\text{ lower degree terms},
\]
for all $n\ge 2$. So, $\alpha$ can not be locally nilpotent, which is a contradiction.

Combining all above cases, we see that $\alpha(A)=0$. 

(2) can be proved similarly noting that any Poisson derivation of $A$ preserves its Poisson center $\pcnt(A)$. 
\end{proof}

\begin{theorem}\label{thm.strongplndcan}
Let $\text{char}(\kk)=0$ and let $A$ be a Poisson algebra.
\begin{enumerate}
      \item Suppose $A$ is an affine domain. If $A$ is (strongly) $\plnd$-rigid, then $A$ is (strongly) Poisson cancellative. 
      \item Suppose $A$ is noetherian and $\pcnt(A)$ is an affine domain. If $A$ is (strongly) $\plndp$-rigid, then $A$ is (strongly) Poisson cancellative.
      \end{enumerate}
\end{theorem}
\begin{proof}
It follows from Theorem \ref{thm.plndcan} and Lemma \ref{lem.plnd}.
\end{proof}

\section{Poisson discriminant and effectiveness}
\label{sec.pdisc}
For noncommutative algebras that are module-finite over their center, the notion of the discriminant has been introduced to compute their automorphism groups \cite{BZ2,CPWZ1,CPWZ2}. In this section we study the discriminant for Poisson algebras and its relation with the \textbf{PZCP} and Poisson automorphism groups. It is a matter of fact that a Poisson algebra appearing in characteristic zero usually does not possess a large Poisson center. Therefore, we introduce the Poisson discriminant from a representation-theoretic point of view following Lu, Wu, and Zhang
\cite[\S 2]{LWZ2}. 

Let $A$ be a Poisson algebra. We denote by $A^\times$ the set of units of $A$. By a {\sf property $\cP$} we mean a property that is invariant under isomorphism within a class of algebras. Herein we generally assume this is the class of Poisson algebras.
\begin{definition}
\label{def.P}
Let $A$ be a Poisson algebra and let $\pcnt = \pcnt(A)$. Let $\cP$ be a property defined for Poisson algebras. We define the following terms for sets/ideals in $A$.
\begin{enumerate}
\item ($\cP$-locus) 
$L_{\cP}(A) := \{ \fm \in \Maxspec(\pcnt) : A/\fm A \text{ has property $\cP$}\}$.
\item ($\cP$-discriminant set) 
$D_{\cP}(A) := \Maxspec(\pcnt) \backslash L_{\cP}(A)$.
\item ($\cP$-discriminant ideal)
$I_{\cP}(A) := 
\bigcap_{\fm \in D_{\cP}(A)} \fm \subset \pcnt$.
\end{enumerate}
In the case that $I_{\cP}(A)$ is a principal ideal, generated by $d \in \pcnt$, then $d$ is called the $\cP$-discriminant of $A$, denoted by $d_{\cP}(A)$. Observe that, if $\pcnt$ is a domain, $d_{\cP}(A)$ is unique up to an element of $\pcnt^\times$.
\end{definition}

\begin{example}
\label{ex.weyl}
Let $A$ be the first Poisson Weyl algebra. It is well-known that $A$ is Poisson simple. Now let $H$ be the Poisson homogenization of $A$ such that $H=\kk[x,y,t]$ with Poisson bracket $\{x,t\}=\{y,t\}=0$ and $\{x,y\}=t^2$. Suppose $\kk$ is algebraically closed of characteristic zero. The Poisson center of $H$ is $\pcnt(H)=\kk[t]$. For each $\alpha \in \kk$, let $m_{\alpha}=(t-\alpha)$ denote the corresponding maximal ideal of $\pcnt(H)$. Then $H/m_{\alpha}H \iso A$ for all $\alpha \in \kk^\times$. On the other hand, $H/m_0 H \iso \kk[x,y]$ with the trivial Poisson bracket, whence not Poisson simple. It follows that if $\cS$ is the property of being Poisson simple, then the $\cS$-locus is $L_\cS(H)=\{m_\alpha : \alpha \neq 0\}$, the $\cS$-discriminant set is $D_{\cS}(H)=\{m_0\}$, and the $\cS$-discriminant exists, that is $d_{\cS}(H)=t$.
\end{example}

\begin{definition}
Let $\cC$ be a class of Poisson $\kk$-algebras.
We say that a property $\cP$ is {\sf $\cC$-stable} if for every Poisson algebra $A$ in $\cC$ and every $n \geq 1$, $I_{\cP}(A \tensor \kk[t_1,\hdots,t_n]) =  I_{\cP}(A) \tensor \kk[t_1,\hdots,t_n]$ as ideals of $\pcnt(A) \tensor \kk[t_1,\hdots,t_n]$.
If $\cC$ is a singleton $\{A\}$, we say $\cP$ is {\sf $A$-stable}.
On the other hand, if $\cC$ is the collection of all Poisson $\kk$-algebras with affine Poisson center over $\kk$, we say $\cP$ is {\sf stable}.
\end{definition}

The following lemma is obvious from Definition \ref{def.P}.

\begin{lemma}
\label{lem.prop}
Let $\cP$ be a property.
If $\phi:A \rightarrow B$ is an isomorphism of Poisson algebras,
then $\phi$ preserves the $\cP$-locus, $\cP$-discriminant set,
$\cP$-discriminant ideal, and, if it exists,
the $\cP$-discriminant.
\end{lemma}

Most of our results will require that $\cP$ is a stable property.
The next lemma, which generalizes \cite[Lemma 6.1]{LWZ2}, 
allows us to work with any property when $\kk$ is algebraically closed.

\begin{lemma}\label{lemm.properP}
Let $\kk$ be algebraically closed. Then any property $\cP$ is stable. 
\end{lemma}
\begin{proof}
Let $A$ be any Poisson algebra over $\kk$ with affine Poisson center $\pcnt$. By Hilbert's Nullstellensatz, there is a natural projection $\pi: \Maxspec(\pcnt[t_1,\dots,t_n])\twoheadrightarrow \Maxspec(\pcnt)$ such that 
$A[t_1,\dots,t_n]/\fm A[t_1,\dots,t_n]\cong A/\pi(\fm)A$
as Poisson algebras over $\kk$ for any ideal $\fm \in \Maxspec(\pcnt[t_1,\dots,t_n])$. 
It now follows in a similar way to \cite[Lemma 6.1]{LWZ2} that $I_{\cP}(A \tensor \kk[t_1,\hdots,t_n]) =  I_{\cP}(A) \tensor \kk[t_1,\hdots,t_n]$ and so $\cP$ is a stable property.
\end{proof}

Our next result is a generalization of Lemma \ref{lemm.cancel} using the Poisson discriminant.

\begin{lemma}
\label{lemm.dretract}
Let $A$ be a Poisson algebra with affine Poisson center $\pcnt(A)$. Let $\cP$ be a stable property and assume that the $\cP$-discriminant $d = d_{\cP}(A)$ of $A$ exists. 
\begin{enumerate}
\item Suppose $A$ is an affine domain. If $\pmlds(A[t])=A$, then $A$ is Poisson retractable.
\item Suppose $A$ is an affine domain. If $A$ is strongly $\plndds$-rigid, then $A$ is strongly Poisson retractable.
\item Suppose $\pcnt(A)$ is a domain. If $\mlds(\pcnt(A)[t])=\pcnt(A)$, then $A$ is $\pcnt$-retractable.
\item Suppose $\pcnt(A)$ is a domain. If $\pcnt(A)$ is strongly $\lndds$-rigid, then $A$ is strongly $\pcnt$-retractable.
\end{enumerate}
\end{lemma}
\begin{proof}
We prove (4). The proofs of (1)-(3) are similar.

Let $\phi:A[t_1,\hdots,t_n] \xrightarrow{\sim} B[s_1,\hdots,s_n]$ be an isomorphism for some Poisson algebra $B$ and some $n \geq 1$. Note that the property $\cP$ is stable and the $\cP$-discriminant ideal is invariant under any Poisson isomorphism by Lemma \ref{lem.prop}.
Write $1$ as the identity element in $\kk[t_1,\dots,t_n]$. Hence,
\begin{align*}
\phi(d \tensor 1) 
	&= \phi( (d) \tensor \kk[t_1,\hdots,t_n] )
	= \phi( I_{\cP}(A[t_1,\hdots,t_n]) ) \\
    &= I_{\cP}(B[s_1,\hdots,s_n])
    = I_{\cP}(B) \tensor \kk[s_1,\hdots,s_n].
\end{align*}
Thus, $I_{\cP}(B)=(d')$ where $d'$ is the $\cP$-discriminant of $B$.
Moreover, the above computation implies that
$\phi(d \tensor 1) =_{\pcnt(B[s_1,\hdots,s_n])^\times} d' \tensor 1'$,
where $1'$ is the identity element in $\kk[s_1,\hdots,s_n]$.
However, 
$$\pcnt(B[s_1,\hdots,s_n])\cong \pcnt(A[t_1,\dots,t_n])\cong \pcnt(A)[t_1,\dots,t_n]$$ 
is a domain, so $\pcnt(B[s_1,\hdots,s_n])^\times=\pcnt(B)^\times$.
It follows that $\phi(d)=_{\pcnt(B)^\times}d'$. Since $\pcnt$ is $\lndds$-rigid, we have
\[\phi(\pcnt(A)) 
	= \phi(\mlds(\pcnt(A) \tensor \kk[t_1,\hdots,t_n]))
	= \ml_{d'}^*(\pcnt(B) \tensor \kk[s_1,\hdots,s_n])
    \subseteq \pcnt(B),
\]
where the last inclusion follows from Lemma \ref{lem.ddx}. By Lemma \ref{lem.subalg}, $\phi(\pcnt(A))=\pcnt(B)$ and $A$ is strongly $\pcnt$-retractable.
\end{proof}

Now we can generalize our previous Theorem \ref{thm.plndcan} by applying the Poisson discriminant. 
\begin{theorem}
\label{thm.ZLPNDretract}
Let $A$ be a Poisson algebra with affine Poisson center $\pcnt=\pcnt(A)$.
Let $\cP$ be a stable property and assume that the $\cP$-discriminant $d = d_{\cP}(A)$ of $A$ exists. 
\begin{enumerate}
\item Suppose $A$ is an affine domain. If $\pmlds(A[t])=A$, then $A$ is Poisson cancellative.
\item Suppose $A$ is an affine domain. If $A$ is strongly $\plndds$-rigid, then $A$ is strongly Poisson cancellative.
\item Suppose $A$ is noetherian and $\pcnt$ is a domain. If $\mlds(\pcnt[t])=\pcnt$, then  $A$ is Poisson cancellative.
\item Suppose $A$ is noetherian and $\pcnt$ is a domain. If $\pcnt$ is strongly $\lndds$-rigid, then $A$ is strongly Poisson cancellative.
\end{enumerate}
\end{theorem}
\begin{proof}
(1)-(2) are from Lemma \ref{lemm.dretract}(1)-(2) and Lemma \ref{lemm.retrdc}(1).

(3)-(4) are from Lemma \ref{lemm.dretract}(3)-(4) and Theorem \ref{thm.dectcan}(3).
\end{proof}

\begin{corollary}
Let $\kk$ be algebraically closed of characteristic zero and let $A$ be a Poisson algebra with affine Poisson center $\pcnt=\pcnt(A)$. Let $\cP$ be any property and assume that the $\cP$-discriminant $d = d_{\cP}(A)$ of $A$ exists. 
\begin{enumerate}
\item Suppose $A$ is an affine domain. If $A$ is (strongly) $\plndds$-rigid, then $A$ is (strongly) Poisson cancellative.
\item Suppose $A$ is noetherian and $\pcnt$ is a domain. If $\pcnt$ is (strongly) $\lndds$-rigid, then $A$ is (strongly) Poisson cancellative.
\end{enumerate}
\end{corollary}
\begin{proof}
Since $\kk$ is algebraically closed, the property $\cP$ is stable by Lemma \ref{lemm.properP}. Then the result follows from Theorem \ref{thm.ZLPNDretract} and Lemma \ref{lem.plnd}(1).
\end{proof}

\begin{remark}\label{rem.d}
Let $A$ be a Poisson algebra with affine Poisson center $\pcnt(A)$, and $\cP$ be a stable property such that the $\cP$-discriminant $d = d_{\cP}(A)$ of $A$ exists. Applying the arguments in Lemma \ref{lemm.dretract}, we can show that the following hold. 
\begin{enumerate} 
     \item If $A$ is (strongly) $\plndds$-rigid, then it must be (strongly) $\plnds$-rigid.
    \item If either $A$ is (strongly) $\plnddps$-rigid or $\pcnt(A)$ is (strongly) $\lndds$-rigid, then $A$ is (strongly) $\plndps$-rigid. 
\end{enumerate}
\end{remark}

According to \cite[\S5]{BZ2}, effectiveness of the discriminant controls $\lndh$-rigidity and hence solves the $\textbf{ZCP}$. For Poisson algebras, we will see that effectiveness of the Poisson discriminant plays the same role in the \textbf{PZCP}. 

\begin{definition}
Let $A$ be an affine (Poisson) commutative domain and suppose that $Y=\bigoplus_{i=1}^n\kk x_i$ generates $A$ as an algebra. An element $0 \neq f \in A$ is called {\sf (Poisson) effective} if the following conditions hold.
\begin{enumerate}
\item There is a (Poisson) $\NN$-filtration $F$ on $A$ such that $\gr_F(A)$ is a domain. With this filtration we define the degree of elements in $A$, denoted by $\deg_A$.
\item For every $\NN$-filtered (Poisson) commutative algebra $T$ with $\gr T$ being an $\NN$-graded (Poisson) domain and for every testing subset $\{y_1,\hdots,y_n\} \subset T$ satisfying
\begin{enumerate}
\item it is linearly independent in the quotient $\kk$-module $T/\kk 1_T$, and
\item $\deg y_i \geq \deg x_i$ for all $i$ and $\deg y_{i_0} > \deg x_{i_0}$ for some $i_0$,
\end{enumerate}
there is a presentation of $f$ of the form $f(x_1,\hdots,x_n)$ in the free algebra $\kk[x_1,\hdots,x_n]$, such that either $f(y_1,\hdots,y_n)$ is zero or $\deg_T f(y_1,\hdots,y_n) > \deg_A f$.
\end{enumerate}
\end{definition}

\begin{remark}
\begin{enumerate}
    \item Our definition of effectiveness borrows from the definition for noncommutative algebras \cite[Definition 5.1]{BZ2}. The key difference there is that the testing algebra should be PI. In our case, that is not necessary because we assume that every algebra is commutative.
    \item There is another concept, called ``dominating", see \cite[Definition 4.5]{BZ2} or \cite[Definition 2.1(2)]{CPWZ1}, which is slightly different from the definition of effectiveness. In this paper, we do not state this concept explicitly.
    \item In many applications of Poisson discriminant we do not need the filtration appearing in the definition of effectiveness to be a Poisson filtration. So we will just use the effectiveness of Poisson discriminant for commutative algebras.
\end{enumerate}
\end{remark}

\begin{example}
\label{ex.kx}
There are some simple examples of effective elements in polynomial algebras. 
\begin{enumerate}
\item For the polynomial algebra $\kk[x]$ it is observed in \cite[Example 2.8]{LWZ1} that any nonzero element $f \in \kk[x]$ is effective. 
In a similar way it is Poisson effective.
\item Let $A=\kk[x_1,\dots, x_n]$ be a
Poisson algebra.
A monomial $x_1^{b_1}\cdots x_n^{b_n}$ in $A$ is said to have degree component-wise less than 
$x_1^{a_1}\cdots x_n^{a_n}$ if $b_i\le a_i$ for all $i$ and $b_{i_0}<a_{i_0}$ for some $i_0$. We write $f=cx_1^{a_1}\cdots x_n^{a_n}+(\text{cwlt})$ if $f-cx_1^{a_1}\cdots x_n^{a_n}$ is a linear combination of monomials with degree compoent-wise less than $x_1^{a_1}\cdots x_n^{a_n}$. 
If $f=cx_1^{a_1}\cdots x_n^{a_n}+(\text{cwlt})$ satisfies $c\neq 0$ and $a_1\cdots a_n>0$, then it follows from \cite[Lemma 2.2(1)]{LWZ1} that $f$ is effective in $A$.
\end{enumerate}
\end{example}

The next result 
is parallel to the corresponding result for discriminants of noncommutative (associative) algebras.
We omit the proof and refer the reader to \cite{BZ2,CPWZ1,CPWZ2} for details. 

\begin{theorem}
Let $A$ be an affine Poisson algebra with generating subspace $Y=\bigoplus_{i=1}^n \kk x_i$, and the associated graded algebra $\gr_Y A$ is a connected graded (Poisson) domain. Let $\cP$ be a stable property and assume that the $\cP$-discriminant $d=d_{\cP}(A)$ of $A$ exists. If $d$ is (Poisson) effective in $A$, then $\Aut_P(A)\subset \text{GL}(Y\oplus \kk)$ is an algebraic group that fits into an exact sequence 
\[
\xymatrix{
1\ar[r] &(\kk^\times)^r\ar[r] &\Aut_P(A)\ar[r] & S\ar[r] &1
}
\]
for some finite group $S$.
\end{theorem}

\begin{example}
Let $A=\mathbb C[x,y,z]$ be the Poisson algebra whose Poisson bracket is of Jacobian form given by the potential $f=x^2-yz$ such that
\[
\{x,y\} = f_z=-y, \quad
\{y,z\} = f_x=2x, \quad
\{z,x\} = f_y=-z.
\]
Since $f$ is homogeneous with isolated singularity at the origin, $\pcnt(A)=\mathbb C[f]$ by \cite[Proposition 4.2]{Pcnt}. Let $\cP$ be the property such that the symplectic foliation appearing in the maximal spectrum of the corresponding Poisson algebra has no 0 dimensional skeletons. It is clear that $A/(f-\lambda)$ has property $\cP$ for any $\lambda\in \mathbb C$ if and only if it has no maximal ideal that is also a Poisson ideal. One checks that the zero ideal is the only maximal ideal of $A$ that is also a Poisson ideal. Hence the $\cP$-discriminant of $A$ is given by $f$. Note $f$ 
is neither effective nor Poisson effective in $A$ since $\Aut_P(A)$ contains Nagata's wild automorphism \cite{nagata} defined by
\[
\sigma(x)=x+(x^2-yz)z,\quad
\sigma(y)=y+2(x^2-yz)x+(x^2-yz)^2z,\quad
\sigma(z)=z.
\]
On the contrary, $f$ is both effective and Poisson effective in the Poisson center $\pcnt(A)=\mathbb C[f]$, which implies that $A$ is strongly Poisson cancellative (see Theorem \ref{thm.prigid}(2) later).  
\end{example}

\begin{lemma}
\label{lemm.prigid}
Let $A$ be a Poisson algebra with affine Poisson center $\pcnt=\pcnt(A)$. Let $\cP$ be a stable property and assume that the $\cP$-discriminant $d = d_{\cP}(A)$ of $A$ exists. 
\begin{enumerate}
    \item Suppose $A$ is an affine domain. If $d$ is effective (respectively, Poisson effective) in $A$, then $A$ is strongly $\plndds$-rigid. 
     \item Suppose $\pcnt$ is a domain. If $d$ is effective in $\pcnt$, then $\pcnt$ is strongly $\lndds$-rigid. 
\end{enumerate}
\end{lemma}
\begin{proof}
We will only prove (1). The proof of (2) is similar. Suppose $A$ is not strongly $\plndds$-rigid. Let $Y=\bigoplus_{i=0}^n \kk x_i$ be a minimal generating space of $A$.  Then there is $(\partial_i)_{i=0}^\infty\in \plnddh(A[t_1,\dots,t_n])$ for some $n\ge 1$ such that $x_{i_0}\not\in \ker \partial$ for some $1\le i_0\le n$. By Lemma \ref{lem:iterative}(1,4), we have $G\in \Aut_P(A[t_1,\dots,t_n][t])$ satisfying $G(d)=d$ and 
\[
G(x_i)=x_i+\sum_{j\ge 1} \partial_j(x_i)t^j.
\]
Since $d$ is (Poisson) effective in $A$, it implies that $A $ has a (Poisson) $\mathbb N$-filtration which naturally induces a filtration on its extension $T:=A[t_1,\dots,t_n][t]$ by assigning arbitrary degrees on $t_1,\dots,t_n$ and $t$. It is easy to see that once $\deg(t_1),\dots,\deg(t_n),\deg(t) \gg 0$, the $\mathbb N$-filtered (Poisson) algebra $T$ has a set $\{G(x_1),\cdots,G(x_n)\}$ where $\deg_T G(x_i)\ge \deg_A x_i$ and in particular $\deg_T G(x_{i_0})> \deg_A x_{i_0}$. But by (Poisson) effectiveness of $d$, we must have 
\[
\deg_A d<\deg_T d\left(G(x_1),\dots,G(x_n)\right)=\deg_T G(d)=\deg_T d=\deg_A d,
\]
a contradiction. Hence $A$ is strongly $\plndds$-rigid. 
\end{proof}

\begin{theorem}
\label{thm.prigid}
Let $A$ be a Poisson algebra with affine Poisson center $\pcnt=\pcnt(A)$. Let $\cP$ be a stable property and assume that the $\cP$-discriminant $d = d_{\cP}(A)$ of $A$ exists. 
\begin{enumerate}
    \item Suppose $A$ is an affine domain. If $d$ is effective (respectively, Poisson effective) in $A$, then $A$ is strongly Poisson cancellative. 
     \item Suppose $A$ is noetherian and $\pcnt$ is a domain. If $d$ is effective in $\pcnt$, then $A$ is strongly Poisson cancellative. 
\end{enumerate}
\end{theorem}
\begin{proof}
The result follows from Lemma \ref{lemm.prigid} and Theorem 
\ref{thm.ZLPNDretract}(2,4).
\end{proof}

\begin{example}
\label{ex.jacdet}
Let $A=\kk[x,y,z]$ be a polynomial algebra over an algebraically closed field $\kk$ of characteristic zero. Suppose $A$ has a Jacobian bracket given by some homogeneous potential $f\in \kk[x,y,z]$ such that 
$\{x,y\} = f_z$, $\{y,z\} = f_x$, $\{z,x\} = f_y$.
Suppose $f$ has isolated singularities. That is, suppose $\kk[x,y,z]/(f_x,f_y, f_z)$ is a finite-dimensional $\kk$-algebra, or equivalently, $(f_x,f_y,f_z)$ is a regular sequence of length 3 in $\kk[x,y,z]$. By \cite[Proposition 4.2]{Pcnt} we have $\pcnt(A)=\kk[f]$. Let $\cP$ be the property of not having finite-dimensional simple Poisson modules. Then for any $\lambda\in \kk$, $A/(f-\lambda)$ has property $\cP$ if and only if no maximal ideal of $A/(f-\lambda)$ is a Poisson ideal by \cite[Lemma 3.1]{Jpoisson}. Moreover, any maximal ideal $(x-x_0,y-y_0,z-z_0)$ for any $x_0,y_0,z_0\in \kk$ is a Poison ideal if and only if 
$(x_0,y_0,x_0)\in \VV\left(f_x,f_y,f_z\right)=:S$,
which is a finite set of points in $\mathbb A^3$. Hence the $\cP$-discriminant of $A$ is given by
\[
d=\prod_{(a,b,c)\in S}\left(f(x,y,z)-f(a,b,c)\right)
\]
is a nonzero element in $\pcnt(A)=\kk[f]$, which is effective in $\pcnt(A)$ by Example \ref{ex.kx}(1). So $A$ is always strongly Poisson cancellative according to Theorem \ref{thm.prigid}(2).
\end{example}

\begin{example}\label{ex.Skly3}
Let $S$ be a three-dimensional Sklyanin algebra over an algebraically closed field $\kk$ of characteristic zero that is module-finite over its center $Z$. Let $(E,\mathscr L,\sigma)$ be the associated geometric data of $S$. It is well-known that $S$ has a central regular element $g$ of degree three and $S/gS\cong B(E,\mathscr L,\sigma)=:B$, the twisted homogeneous coordinate ring of $(E,\mathscr L,\sigma)$. When $S$ is PI, the automorphism $\sigma$ of the elliptic curve $E$ has finite order and $B$ is a GK-dimension two domain that is module-finite over its center $Z(B)\cong Z/gZ$. It is proved in \cite[Proposition 5.5(1)]{WWY1} that there exists a unique nonzero Poisson structure (up to scalars) on $Z$ if we assume $g$ to be in the corresponding Poisson center. We can show that  $Z$ is strongly Poisson cancellative in this case. We observe that $\pcnt(Z)=\kk[g]$. Since $Z/gZ\cong Z(B)$ is a Poisson domain of Krull dimension two, $Z/gZ$ has trivial Poisson center by Corollary \ref{cor.cancel1}. Then, by passing to $\pcnt(Z/gZ)=\kk$, an easy induction on degree of homogeneous elements in $\pcnt(Z)$ yields the result. Now let $\cP$ be the property of having no or at least two zero dimensional symplectic core skeletons in the symplectic core stratification of the corresponding maximal spectrum. Then \cite[Theorem 1.3]{WWY1} shows that the $\cP$-discriminant of $Z$ is exactly given by $g$. Hence $Z$ is strongly Poisson cancellative by Theorem \ref{thm.prigid}(2) and Example \ref{ex.kx}(1).
\end{example}

\section{Some remarks and questions}\label{S:RQ}
In this last section, we provide some remarks and questions for future projects. First of all, we summarize the related concepts that are introduced in our paper to the \textbf{PZCP}.
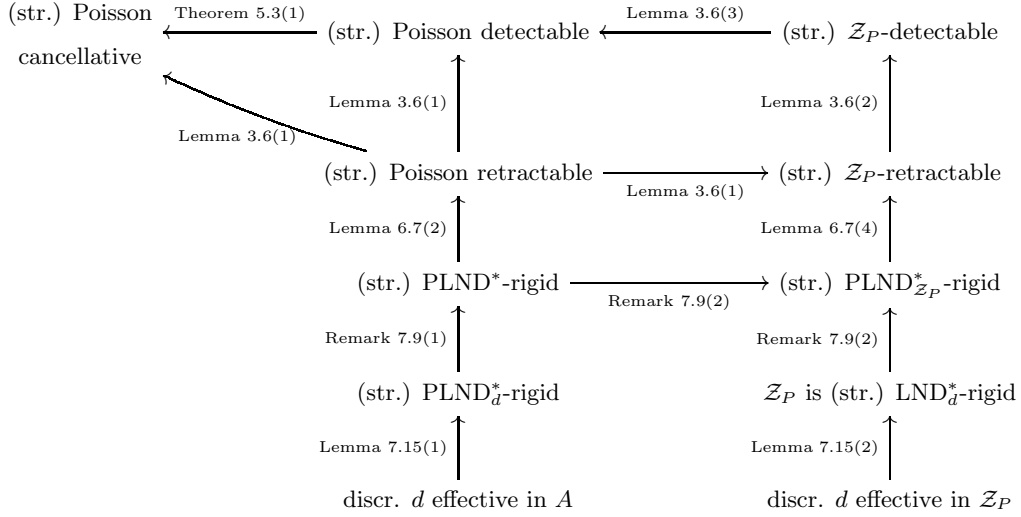
\[\small
\xymatrix{
{\begin{matrix}\text{(str.) Poisson} \\ \text{cancellative}\end{matrix}}
&&\text{(str.) Poisson detectable}\ar[ll]_-{\text{Theorem \ref{thm.dectcan}(1)}} && \text{(str.) $\pcnt$-detectable}\ar[ll]_-{\text{Lemma \ref{lemm.retrdc}(3)}}\\
&&\text{(str.) Poisson retractable}\ar@/^/[ull]^-{\text{Lemma \ref{lemm.retrdc}(1)}}\ar[u]^-{\text{Lemma \ref{lemm.retrdc}(1)}}\ar[rr]_-{\text{Lemma \ref{lemm.retrdc}(1)}} && \text{(str.) $\pcnt$-retractable}\ar[u]^-{\text{Lemma \ref{lemm.retrdc}(2)}}\\
&&\text{(str.) $\plnds$-rigid}\ar[u]^-{\text{Lemma \ref{lemm.cancel}(2)}}\ar[rr]_-{\text{Remark \ref{rem.d}(2)}} && \text{(str.) $\plndps$-rigid}\ar[u]^-{\text{Lemma \ref{lemm.cancel}(4)}}\\
&&\text{(str.) $\plndds$-rigid}\ar[u]^-{\text{Remark \ref{rem.d}(1)}}   && \text{$\pcnt$ is (str.) $\lndds$-rigid}\ar[u]^-{\text{Remark \ref{rem.d}(2)}} \\
&&\text{discr. $d$ effective in $A$}\ar[u]^-{\text{Lemma \ref{lemm.prigid}(1)}}  &&  \text{discr. $d$ effective in $\pcnt$}\ar[u]^-{\text{Lemma \ref{lemm.prigid}(2)}}\\
}
\]

When $A$ is a Poisson algebra with trivial Poisson bracket, 
\textbf{PZCP} for $A$ becomes the ordinary \textbf{ZCP} for $A$. This yields our first question.

\begin{question}
Let $A$ be a Poisson algebra. If $A$ is (strongly, universally) cancellative as a commutative algebra, is $A$ always (strongly, universally) cancellative in the sense of Poisson algebras?
\end{question}

The inverse implication of the above question does not hold. In Example \ref{nonZCP}(1), the coordinate ring of the real sphere $\mathbb R[x,y,z]/(x^2+y^2+z^2-1)$ is not cancellative. But we can endow it with a Jacobian bracket given by the potential $f=x^2+y^2+z^2$ such that $\{x,y\}=2z$, $\{y,z\}=2x$, and $\{z,x\}=2y$. One can show the resulting Poisson structure on the coordinate ring yields the trivial Poisson center. Hence it is universally Poisson cancellative by Theorem \ref{thm.univ}. So this yields our second question.

\begin{question}
Let $A$ be a commutative algebra that is not cancellative. Can we always endow $A$ with a Poisson bracket that makes it Poisson cancellative?
\end{question}

Our Theorem \ref{thm.iso} is a Poisson analogue of the isomorphsm lemma for connected graded algebras generated in degree one \cite[Theorem 1]{BZ1}. Note that the original isomorphism lemma for connected graded algebras has been generalized for graded path algebras \cite{Gpath}. We expect that our result can be extended to graded path Poisson algebras as well.

\begin{question}
Let $A$ and $B$ be two $\NN$-graded Poisson algebras that are finitely generated in degree 0 and 1. If $A\cong B$ as ungraded Poisson algebras, does $A\cong B$ as graded Poisson algebras? 
\end{question}

In Example \ref{ex.jacdet}, we showed that any polynomial algebra of three variables with Jacobian bracket is strongly Poisson cancellative if the potential related to the Jacobian form is homogeneous with isolated singularities. We would like to know if the cancellation property holds more generally.
\begin{question}
Is any polynomial Poisson algebra in three variables with nontrivial Jacobian bracket Poisson cancellative?
\end{question}

Moreover, every unimodular Poisson algebra on $\kk[x,y,z]$ has Jacobian bracket \cite[Theorem 5]{PRZ}. For noncommutative algebras, the \textbf{ZCP} was asked for Artin-Schelter regular algebras in \cite[Question 0.2]{BZ2}.  For a connected graded algebra, the (skew) Calabi-Yau property is equivalent to Artin-Schelter regularity \cite[Lemma 1.2]{RRZ}. Note that unimodularity of Poisson algebras is an analogue of Calabi-Yau property for noncommutative algebras. We refer the interested reader to \cite{LWZ4} for the notion of unimodularity and its connection to Calabi-Yau algebras.
Therefore, we are interested in the \textbf{PZCP} for unimodular Poisson  algebras.

\begin{question}
\label{ques.unimod}
Let $A=\CC[x_1,\dots,x_n]$ be a unimodular complex Poisson polynomial algebra with nontrivial Poisson bracket. When is $A$ Poisson cancellative? 
\end{question}

In Example \ref{ex.Skly3}, we show that for a three-dimensional Sklyanin algebra that is module-finite over its center, any nontrivial Poisson bracket on the center makes it strongly Poisson cancellative provided the canonical central regular element $g$ is in the Poisson center. A similar phenomena occurs with any four-dimensional Sklyanin algebra $S$ that is module-finite over its center $Z$. It is proved in \cite[Proposition 7.3]{WWY2} that the Poisson structure on $Z$ is uniquely determined as long as the two canonical central regular elements $g_1,g_2$ of $S$ are in the Poisson center of $Z$. 

\begin{question}
Let $S$ be a four-dimensional Sklyanin algebra that is module-finite over its center $Z$. Suppose $Z$ has a nontrivial Poisson structure where the two canonical central regular elements $g_1,g_2$ of $S$ are in the Poisson center of $Z$. Is $Z$ strongly Poisson cancellative?
\end{question}

For any Poisson algebra $A$, there is a notion of Poisson universal enveloping algebra $U(A)$, whose representation category is Morita equivalent to the category of Poisson modules over $A$. Now we recall the definition of $U(A)$ from \cite{OH3}. Let $A$ be a Poisson algebra and let $m_A =\{m_a\,|\,a\in A\}$ and $h_A =\{h_a\,|\,a\in A\}$ be two copies of the vector space $A$ endowed with two $\kk$-linear isomorphisms $m:A\to m_A: a\mapsto m_a$ and $h:A\to h_A: a \mapsto h_a$ for any $a\in A$. Then the Poisson universal enveloping algebra $U(A)$ is an associative algebra over $\kk$, with an identity 1, generated by $m_A$ and $h_A$ with relations, for any $x,y\in A$,

\begin{align*}
m_{xy}&= m_xm_y,\\
h_{\{x,y\}}&=h_xh_y-h_yh_x,\\
h_{xy}&=m_yh_x +m_xh_y,\\
m_{\{x,y\}} &=h_xm_y-m_yh_x =[h_x,m_y],\\
m_1 &= 1.
\end{align*}

\begin{question}
Let $A$ be any affine Poisson algebra, and $U(A)$ be its Poisson universal enveloping algebra. What is the relationship between $U(A)$ being cancellative and $A$ being Poisson cancellative?
\end{question}

In practice many Poisson structures can be derived from the process of  semiclassical limits, for instance see \cite{Goodsemi}, which we will recall now. Suppose $R$ is a torsionfree $\kk[\hbar]$-algebra such that $R/\hbar R$ is a commutative $\kk$-algebra. Denote the specialization map $\pi: R\to R/\hbar R$. The algebra $R/\hbar R$ equipped with the Poisson bracket: 
\[ \left\{\pi(a),\, \pi(b) \right\} := \pi\left(\frac{ \left[a,b\right]}{\hbar}\right)\ \text{for all $a,b\in R$,}\]
is called the semiclassical limit of the family of (noncommutative) algebras $\left( R_\alpha \right)_{\alpha \in \kk}$, where $R_\alpha:=R/(\hbar-\alpha)R$.

\begin{question}
Is the Poisson cancellation property of $R/\hbar R$ related to the cancellation property of the $R_\alpha$?
\end{question}

\bibliographystyle{amsplain}
\providecommand{\bysame}{\leavevmode\hbox to3em{\hrulefill}\thinspace}
\providecommand{\MR}{\relax\ifhmode\unskip\space\fi MR }
\providecommand{\MRhref}[2]{%
  \href{http://www.ams.org/mathscinet-getitem?mr=#1}{#2}
}
\providecommand{\href}[2]{#2}

\end{document}